\documentclass{amsart}

\usepackage[latin1]{inputenc}
\usepackage[final]{pdfpages} 
\usepackage{amsmath,amssymb,mathtools}
\usepackage[T1]{fontenc}
\usepackage{amsthm}
\usepackage[all]{xy}
\usepackage{stmaryrd}
\usepackage[hidelinks]{hyperref}
\usepackage{faktor}
\usepackage{comment}
\usepackage{enumerate}

\newtheorem{theo}{Theorem}[section]
\newtheorem{prop}[theo]{Proposition}

\newtheorem{lemma}[theo]{Lemma}
\newtheorem{cor}[theo]{Corollary}

\theoremstyle{definition}
\newtheorem{art}[theo]{}
\newtheorem{defi}[theo]{Definition}

\theoremstyle{remark}
\newtheorem{example}[theo]{Example}
\newtheorem{rem}[theo]{Remark}
\newcommand{\A}{\mathcal{A}}

\newcommand{\B}{\mathcal{B}}

\renewcommand{\H}{\mathcal{H}}

\newcommand{\N}{\mathbb{N}}
\renewcommand{\O}{\mathcal{O}}
\newcommand{\Oo}{\mathcal{O}^\circ}

\newcommand{\Qp}{\mathbb{Q}_p}

\newcommand{\R}{\mathbb{R}}

\newcommand{\X}{\mathcal{X}}

\newcommand{\Z}{\mathbb{Z}}

\newcommand{\Zp}{\mathbb{Z}_p}

\renewcommand{\and}{\text{and}}
\newcommand{\im}{\text{im}}

\newcommand{\st}{ \ \big| \ }

\newcommand{\cha}{\text{char}}

\newcommand{\Spec}{\text{Spec}}

\newcommand{\val}{\ensuremath{\sqrt{|k^{*} |}}}

\newcommand{\Ao}{\mathcal{A}^\circ}
\newcommand{\Aoo}{\mathcal{A}^{\circ \circ}}
\newcommand{\Ah}{\check{\mathcal{A}}}
\newcommand{\At}{\widetilde{\mathcal{A}}}
\newcommand{\Atp}{\widetilde{\mathcal{A}}^+}
\newcommand{\Atpr}{\widetilde{\mathcal{A}}^+_r}
\newcommand{\Aor}{\mathcal{A}^\circ_r}
\newcommand{\Aoor}{\mathcal{A}^{\circ \circ}_r}

\newcommand{\Atr}{\widetilde{\mathcal{A}}^+_r}
\newcommand{\Aos}{\mathcal{A}^\circ_s}

\newcommand{\Ared}{{\A_{\rm red}}}

\newcommand{\Atot}{\widetilde{\A}^+_{\text{tot}} }

\newcommand{\Bh}{\check{\mathcal{B}}}
\newcommand{\Bo}{\mathcal{B}^{\circ}}
\newcommand{\Boo}{\mathcal{B}^{\circ \circ}}
\newcommand{\Bor}{\mathcal{B}^\circ_r}
\newcommand{\Bt}{\widetilde{\mathcal{B}}}
\newcommand{\Btot}{\widetilde{\B}^+_{\text{tot}} }
\newcommand{\Btp}{\widetilde{\mathcal{B}}^+}
\newcommand{\Btr}{\widetilde{\mathcal{B}}^+_r}

\newcommand{\Df}{\mathfrak{D}}

\newcommand{\coker}{\text{coker}}

\newcommand{\GOX}{\Gamma(X,\O_X)}

\newcommand{\Ker}{\text{Ker}}

\newcommand{\kX}{k\langle \zeta  \rangle}
\newcommand{\kXX}{k\langle\langle \zeta \rangle\rangle}

\newcommand{\knX}{k\langle \varepsilon_n^{-1}\zeta \rangle}

\renewcommand{\O}{\mathcal{O}}

\newcommand{\Quot}{\text{Quot}}

\newcommand{\red}{\text{red}}

\newcommand{\rK}{\sqrt{|k^*|}}
\newcommand{\rk}{\sqrt{|k^*|}}
\newcommand{\Rmn}{R\langle T_1,\ldots ,T_m\rangle \llbracket S_1, \ldots , S_n \rrbracket}
\newcommand{\RmnT}{R\langle T_1,\ldots ,T_m\rangle \llbracket S_1, \ldots , S_n \rrbracket}

\newcommand{\supn}[1]{|#1|_{\sup}}

\newcommand{\tk}{\widetilde{k}}

\newcommand{\tX}{\widetilde{X}}

\newcommand{\Uf}{\mathfrak{U}}

\newcommand{\Xf}{\mathfrak{X}}
\newcommand{\Xg}{\mathfrak{X}_\eta}

\newcommand{\Xt}{\widetilde{X}}
\newcommand{\Xred}{{X_{\rm red}}}

\DeclareMathOperator{\spe}{sp}
\DeclareMathOperator{\Spf}{Spf}

\date{\today}
\begin{document}
\title[Analytic functions on tubes of non-archimedean analytic spaces]{Analytic functions on tubes of non-archimedean analytic spaces}

\author{Florent Martin with a joint appendix with Christian Kappen}

\address{Florent Martin, Fakult\"{a}t f\"{u}r Mathematik, Universit\"{a}t Regensburg, 93040 Regensburg, Germany}
\email{florent.martin@mathematik.uni-regensburg.de}
\urladdr{http://homepages.uni-regensburg.de/$\sim$maf55605/}

\address{Christian Kappen}
\email{chrkappen@me.com}

\begin{abstract}
Let $k$ be a discretely valued non-archimedean field.
We give an explicit description of analytic functions whose norm is bounded by a given real number $r$ on tubes of reduced $k$-analytic spaces associated to special formal schemes
(those include $k$-affinoid spaces as well as open polydiscs).
As an application we study the connectedness of these tubes.
This generalizes (in the discretely valued case) a result of Siegfried Bosch.
We use as a main tool a result of Aise Johan de Jong relating formal and analytic functions on special formal schemes and a generalization of de Jong's result which is proved in the joint appendix with Christian Kappen.
\end{abstract}


\maketitle

\setcounter{tocdepth}{1}

\section{Introduction}
Let us temporarily consider a non-archimedean non-trivially valued field $k$.
We will work with $k$-analytic spaces which were introduced by Vladimir Berkovich in \cite{Berko90,Berko93} (and we will always consider strictly $k$-affinoid and strictly $k$-analytic spaces).
If $X$ is a $k$-affinoid space, its ring of analytic functions $\A$ is a $k$-affinoid algebra which has nice algebraic properties.  
A $k$-analytic space is connected if and only if its ring of global analytic functions contains no nontrivial idempotents. 
In concrete situations, if $X$ is a $k$-affinoid space, one can expect to use the nice algebraic properties of its $k$-affinoid algebra $\A$ to study the connectedness of $X$. 
For more general $k$-analytic spaces, it might be difficult to deal with their ring of global analytic functions.
For instance, the ring of global analytic functions of the open unit disc is not Noetherian. 
The starting point of this work is to use generic fibres of special formal schemes to overcome this difficulty in certain situations.

\subsection*{A result of Siegfried Bosch}
Our motivation is a generalization as well as a new proof of a result due to Siegfried Bosch \cite{BoschBemerk} in the discretely valued case.
Let us consider a $k$-affinoid algebra $\A$ and let $X$ be the associated $k$-affinoid space.
Let $x$ be a rigid point of $X$, $\mathfrak{m}_x$ the associated maximal ideal of $\A$  and let $\widetilde{x}$ be the image of $x$ under the reduction map $\red : X \to  \widetilde{X}$ 
(for the definition and properties of the reduction map, we refer to \cite[7.1.5]{BGR} for rigid spaces and to \cite[2.4]{Berko90} for $k$-analytic spaces). 
Let $X_+(x):= \red^{-1}(\widetilde{x})$. 
Following Pierre Berthelot's terminology  \cite[1.1.2]{Berth}, we call $X_+(x)$ the tube of $\widetilde{x}$ in $X$.
S. Bosch proves that if $X$ is distinguished and equidimensional, $X_+(x)$  is connected.
This connectedness result is a corollary of the main result of Bosch's article \cite[Theorem 5.8]{BoschBemerk}, 
which asserts that if $X$ is distinguished and equidimensional, then 
\begin{equation*}
\label{eq:Bosch}
\Gamma(X_+(x), \Oo_X) 
\simeq (\Ao)^{\wedge (t\cdot \Ao + \overset{\circ}{\mathfrak{m}_x})} 
\simeq \varprojlim_n \Ao/(t\cdot \Ao + \overset{\circ}{\mathfrak{m}_x})^n
\end{equation*} 
where  $t\in k$ with $0<|t|<1$, $\Oo_X$ denotes the sheaf of analytic functions $f$ such that $|f|_{\sup} \leq 1$, $\overset{\circ}{\mathfrak{m}_x}:= \mathfrak{m}_x \cap \Ao$ 
and $^\wedge$ denotes the completion with respect to an ideal. \par 

\subsection*{Analytic functions and formal functions}
For the rest of the article, we assume that $k$ is discretely valued, with  non-trivial valuation.
We denote its valuation ring by $R$ and we fix a uniformizer $\pi$.
Following \cite[section 1]{Berko96}, we say that an adic $R$-algebra $A$ is a \emph{special $R$-algebra} if it is isomorphic to a quotient of $\RmnT$ 
equipped with the $(\pi,S_1,\ldots,S_n)$- adic topology. 
Let  $\Xf:= \Spf(A)$ be  its associated formal $R$-scheme. 
Following a construction due to  Berthelot \cite[0.2.6]{Berth} for rigid spaces and extended to $k$-analytic spaces by  Berkovich \cite[section 1]{Berko96}, 
one can associate to $\Xf$ a $k$-analytic space denoted by $\Xg$ called its generic fibre. 
For instance, if $A= \RmnT$, $\Xg \simeq E^m\times B^n$ 
where we denote by $E^m$ (resp. $B^n$) the $m$-dimensional closed unit polydisc (resp. the $n$-dimensional open unit polydisc).
Following the terminology introduced by Christian Kappen in \cite{KappThesis,KappUni}, we say that $\A := A\otimes_Rk$ is a \emph{semi-affinoid $k$-algebra}.
Up to canonical isomorphism, the $k$-analytic space $\Xg$ depends only on the semi-affinoid $k$-algebra $\A$ and we call it a 
\emph{semi-affinoid $k$-analytic space} (this should not be confused with semi-affinoid $k$-spaces in \cite{KappThesis,KappUni}).
So we can functorially associate to a semi-affinoid $k$-algebra a $k$-analytic space.
If $A$ is $R$-flat, one gets a natural injection $A \rightarrow \Gamma(\Xg,  \Oo_{\Xg})$. 
When $A$ is in addition normal, it was proven by A.J. de Jong \cite[7.4.1]{deJongCryst} that 
\begin{equation*}
\label{eq:deJong}
 A \simeq \Gamma(\Xg,  \Oo_{\Xg}).
\end{equation*} 
For the applications we have in mind, we need the following generalization which was already stated without proof in \cite[7.4.2]{deJongCryst}.

\bigskip
\noindent
{\bf Theorem \ref{theo:SpecialIntegral}.}
\emph{ Let $A$ be a reduced special $R$-algebra which is $R$-flat, integrally closed in $A\otimes_Rk$, and let $X$ be the associated $k$-analytic space. 
Then $A \simeq \Gamma(X,  \Oo_{X})$.
}

\bigskip 
\noindent
Let us mention that if $A$ is a special $R$-algebra which is $R$-flat and integrally closed in $A\otimes_Rk$, then $A$ is automatically reduced 
(see the argument in \cite[Remark 2.7]{KappUni}).
Hence one can remove the assumption that $A$ is reduced in the above theorem if necessary.
Theorem \ref{theo:SpecialIntegral} easily follows from the following result which is proved in the appendix with Christian Kappen.

\bigskip
\noindent
{\bf Theorem \ref{theo:bound}.} 
\emph{ Let $\A$ be a reduced semi-affinoid $k$-algebra, and let $X$ be the associated $k$-analytic space. Then 
$\A \simeq \{f \in \Gamma(X,\O_{X}) \st \supn{f}< \infty \} $.
}

\subsection*{Main result}
Actually, the tube $X_+(x)$ is a semi-affinoid $k$-analytic space.
Therefore, it seemed very natural to us to look for a generalization and a more direct proof of Bosch's result \cite[Theorem 5.8]{BoschBemerk} using special formal $R$-schemes, 
semi-affinoid $k$-algebras and de Jong's result as well as its generalization (Theorem \ref{theo:SpecialIntegral}). 
This is the content of this article.
Let us fix $X$  a reduced  semi-affinoid $k$-analytic space.
Let $f_1,\ldots,f_n\in \Gamma(X,\O^\circ_X)$ and let 
\[U:= \{x\in X \st |f_i(x)| <1 \ \forall i=1\ldots n\}.\]
\bigskip
\noindent
{\bf Theorem \ref{theo:Completion}} and {\bf Proposition \ref{prop:Bor}.}
\emph{ With the above notations,}
\begin{equation}
\label{eq:completion}
 \Gamma(X,\O^\circ_X) ^{\wedge (f_1,\ldots f_n)} \simeq \Gamma(U,\Oo_{X}).
 \end{equation}
 More generally, for any positive real number $r$, 
$ \Gamma(X,\O^{\leq r}_X) ^{\wedge (f_1,\ldots f_n)} \simeq \Gamma(U,\O^{\leq r}_{X})$
 where $\O_X^{\leq r}$ is the sheaf of analytic functions $f$ such that $|f| \leq r$. 

\bigskip
\par 
In section \ref{section:ReductionConnectedness} we associate to a semi-affinoid $k$-analytic space $X$ its \emph{canonical reduction} $\tX$ which is  a $\widetilde{k}$-scheme of finite type, and a canonical reduction map $\red : X \to \tX$.
If $X$ is a $k$-affinoid space, then $\red$ coincides with the reduction map of \cite[2.4]{Berko90}. 
We prove the following result:

\bigskip 
\noindent 
{\bf Corollary \ref{cor:Tubes2}.}
\emph{ Let $Z \subset \tX$ be a connected Zariski closed subset.
Then $\red^{-1}(Z)$ is connected.}

\bigskip 
We want to stress that our results hold under the assumption that $k$ is discretely valued, whereas this assumption is not made in \cite{BoschBemerk}. 
We conjecture that Theorem \ref{theo:Completion} holds for any non-trivially valued non-archimedean field $k$ and for any reduced $k$-affinoid space 
(with $(f_1,\ldots f_n)$ replaced by $(t,f_1,\ldots f_n)$ for some $t\in k^*$ with $|t|<1$).
We do not see how this could be done using the techniques of \cite{BoschBemerk}. 
We believe that quasi-affinoid $k$-algebras 
(which are a generalization of special $R$-algebras and semi-affinoid $k$-algebras 
to arbitrary non-archimedean non-trivially valued fields \cite{LR_ring}) 
might be a good framework to tackle this. 
We also want to stress that our results are more general and the proofs simpler than in \cite{BoschBemerk} regarding the following points.
\\
- The proof we give of \eqref{eq:completion} is pretty short: one has to use Theorem \ref{theo:SpecialIntegral} and a certain compatibility between integral 
closure and tensor product for excellent rings (whose use was suggested to us by Ofer Gabber).
\\
- The explicit description of the rings $\Gamma(U,\Oo_{X})$ given in 
Theorem \ref{theo:Completion} holds for any tube whereas in \cite{BoschBemerk} this was proved only for tubes over closed points. 
\\
- For a positive real number $r$, we extend Bosch's result to analytic functions $f$ such that $|f|_{\sup} \leq r$ (Proposition \ref{prop:Bor}).
\\
- Unlike in \cite{BoschBemerk}, we do not assume that $X$ is distinguished or equidimensional. 
Hence, Corollary \ref{cor:Tubes2} answers positively in the discretely valued case the question raised by J\'{e}r\^{o}me Poineau in \cite[Remarque 2.9]{Poi14}.
Let us point out that using Bosch's result, Antoine Ducros has proved \cite[Lemma 3.1.2]{Duc_sa} that when $k$ is algebraically closed and $X$ is equidimensional and reduced, the tube of a connected Zariski closed subset of $\widetilde{X}$ in $X$ is connected;
afterwards in \cite[2.8]{Poi14} J. Poineau (still relying on Bosch's result) proved the same statement for any $k$ and assuming only that $X$ is equidimensional. 
\\
- Our results hold not only for $k$-affinoid spaces, but also for semi-affinoid $k$-analytic spaces.
\subsection*{Organisation of the paper}
In section \ref{section:semi-affinoid}, we give general facts about semi-affinoid $k$-analytic spaces.
In section \ref{section:Analytic functions and formal functions on tubes}, we prove the main result of the article: Theorem \ref{theo:Completion}.
In section \ref{section:ReductionConnectedness} we define and study the canonical reduction of  semi-affinoid $k$-analytic spaces, and apply it to study the connectedness  of tubes.
In section \ref{section: Graded Version} we prove Proposition \ref{prop:Bor} which is the graded version of Theorem \ref{theo:Completion}.
In section \ref{section:additional remarks} we grasp additional remarks about semi-affinoid $k$-analytic spaces.
The joint appendix with Christian Kappen aims to prove Theorem \ref{theo:bound}.

\textbf{Acknowledgements:} 
We would like to thank the referee for a very careful reading leading to a great improvement of this text.
The first author would like to thank Ofer Gabber for pointing out a result of EGA IV which greatly simplified the proof of the main result. 
He also would like to thank 
Antoine Ducros, J\'{e}r\^{o}me Poineau and Michael Temkin for useful comments and conversations about this work.
The first author acknowledges support from SFB 1085 Higher invariants funded by the Deutsche Forschungsgesellschaft (DFG).

\section{Semi-affinoid \texorpdfstring{$k$}{k}-analytic spaces}
\label{section:semi-affinoid}

Following \cite[definition 2.2]{KappUni}, we say that a $k$-algebra $\A$ is a semi-affinoid $k$-algebra if it is of the form $A \otimes_R k$ for some special $R$-algebra $A$.
Equivalently, a semi-affinoid $k$-algebra is a quotient of $\RmnT\otimes_Rk$ for some given integers $m$ and $n$. 
The category of semi-affinoid $k$-algebras is defined as the category whose objects are semi-affinoid $k$-algebras and whose morphisms are $k$-algebra morphisms.
In particular, the category of $k$-affinoid algebras is a full subcategory of the category of semi-affinoid $k$-algebras.
If $\A$ is a semi-affinoid $k$-algebra,  a \emph{special $R$-model} of $\A$ is an $R$-flat special $R$-algebra $A$  such that there exists an isomorphism of $k$-algebras $A\otimes_Rk \simeq \A$.
If $A$ is an $R$-flat special $R$-algebra, one has a natural inclusion 
$A \to \A:=A\otimes_Rk$ and 
through this inclusion, $A$ is identified with a special $R$-model of $\A$. 
One can define \cite[2.3.1]{KappUni} a functor from the category of semi-affinoid $k$-algebras to the category of $k$-analytic spaces. 
If $\A$ is a semi-affinoid $k$-algebra,  its associated $k$-analytic space $X$ is called a semi-affinoid $k$-analytic space.
For any special $R$-model $A$ of $\A$, one has a natural isomorphism 
$X \simeq \Spf(A)_\eta$.
If $f\in \A$, we set 
$|f|_{\sup} := \sup\{ |f(x)| \st x\in X \}$ 
and we define 
\[\Ao  =  \{f \in \A \st |f|_{\sup} \leq 1\}.\]
It is proved in the appendix (Theorem \ref{theo:bound}) that if the semi-affinoid $k$-algebra $\A$ is reduced, one has an isomorphism 
\[ \A \simeq \{f \in \GOX \st \supn{f}< \infty \} .\]
We are particularly interested in the following consequence of this result.

\begin{theo}
\label{theo:SpecialIntegral}
Let $A$ be a reduced special $R$-algebra with associated $k$-analytic space $X$.
Let us assume that $A$ is $R$-flat and integrally closed in $A\otimes_Rk$. 
Then $A \simeq \Gamma(X,  \Oo_{X})$.
\end{theo}

\begin{proof}
Let $\A = A\otimes_Rk$.
By definition, $\A$ is a reduced semi-affinoid $k$-algebra so thanks to Theorem \ref{theo:bound}, one has 
$\Ao\simeq \Gamma(X,\Oo_{X})$. 
According to  \cite[Corollaries 2.10 and 2.11]{KappUni} $\Ao$ is a special $R$-algebra which contains $A$ and the inclusion $A \subset \Ao$ is integral. 
Since by assumption $A$ is integrally closed in $\A$, $A$ is also integrally closed in $\Ao$. 
So $A \simeq \Ao\simeq \Gamma(X,  \Oo_{X})$.
\end{proof}

\begin{cor}
\label{cor:reduced special}
Let $\A$ be a reduced semi-affinoid $k$-algebra with associated $k$-analytic space $X$.
\begin{enumerate}[(i)]
\item The $R$-algebra $\Gamma(X,  \Oo_{X})$ is a reduced special $R$-algebra.
\item If $A$ is a reduced special $R$-algebra such that $A\otimes_Rk \simeq \A$, then $\Gamma(X,  \Oo_{X})$ is isomorphic to the integral closure of $A$ in $A \otimes_R k$.
\end{enumerate}
\end{cor}

\begin{proof}
Let $A$ be a reduced special $R$-model of $\A$, so that $A\otimes_Rk \simeq \A$. 
Let $A'$ be the integral closure of $A$ in $\A$.
Since $A$ is excellent (see \cite{Val75,Val76}), $A'$ is a finite $A$-algebra, so $A'$ is a  reduced special $R$-algebra. 
Since $A'\otimes_Rk \simeq A\otimes_Rk \simeq \A$, the $k$-analytic spaces attached to $A$ and $A'$ are both isomorphic to $X$.
So thanks to Theorem \ref{theo:SpecialIntegral}, $\Gamma(X,  \Oo_{X})\simeq A'$ which is a special $R$-algebra.
This proves (i) and (ii).
\end{proof}

\begin{example}
\label{remark example nonreduced1}
\label{CE1}
The two statements of Corollary \ref{cor:reduced special} do not hold for  non-reduced  semi-affinoid $k$-algebras.
For instance, if $\A  =  (R[[T_1,T_2]] \otimes_R k )/ (T_2^2)$ with associated $k$-analytic space $X$, then any element of 
$\Gamma(X,  \Oo_{X})$  is of the form $f(T_1) + g(T_1) T_2$ 
where $f(T_1) \in R \llbracket T_1 \rrbracket$ and $g(T_1)$ is an arbitrary analytic function on the open unit disc.
We can choose $g(T_1)$  such that 
$g(T_1) \notin R\llbracket T_1 \rrbracket \otimes_R k$ (in mixed characteristic, one could take $g = \log$).
This gives a counterexample to  the statements (i) and (ii) of Corollary \ref{cor:reduced special} (for (i), see \cite[Remark 2.7]{KappUni}).
\end{example}

\begin{cor}
\label{cor:equivalence}
The functor 
\[ \{\text{semi-affinoid}\ k\text{-algebras}\} \to \{ k\text{-analytic spaces}\} \]
is faithful  and its restriction 
\[ \{ \text{reduced semi-affinoid} \ k\text{-algebras}\} \to \{ k\text{-analytic spaces}\} \]
is fully faithful.
\end{cor}

\begin{proof}
If $f : X \to Y$ is a morphism of $k$-analytic spaces induced by a morphism of  semi-affinoid $k$-algebras $\B \to \A$ then $f$ induces 
the following diagram:
\[ \xymatrix{
\B \ar[r] \ar[d]  & \A \ar[d] \\
\Gamma(Y,\O_Y) \ar[r] & \Gamma(X,\O_X)
} \]
whose vertical arrows are injective. 
This proves that the functor is faithful.\par 
To prove that the restriction of the functor to reduced semi-affinoid $k$-algebras is full, let us 
fix a reduced semi-affinoid $k$-algebra $\A$ and let $X$ be its associated $k$-analytic space.
Let $Y$ be the $k$-analytic space associated to another semi-affinoid $k$-algebra $\B$ and let us consider a presentation 
\[\B = (\RmnT\otimes_Rk)/I\] 
giving rise to a closed immersion $Y \hookrightarrow E^m\times B^n$.
If $f : X \to Y$ is a morphism of $k$-analytic spaces then using the composition 
\[ X \xrightarrow[]{f} Y \hookrightarrow E^m\times B^n\]
one gets functions $f_1,\ldots,f_m\in \Gamma(X,\Oo_X)$ and 
$f_{m+1},\ldots,f_{m+n}\in \Gamma(X,\Oo_X)$ with $|f_i(x)| <1$ for all $x\in X$ and $i\in \{m+1,\ldots,m+n\}$ such that 
the functions $f_1,\ldots,f_{m+n}$ induce the morphism $X \to E^m\times B^n$.
Thanks to Theorem \ref{theo:bound}, $f_i \in \Ao$ for $i=1\ldots m+n$. 
Thanks to \cite[Theorem 2.13]{KappUni} there is a unique morphism of semi-affinoid $k$-algebras 
$\RmnT\otimes_Rk \to \A$ sending 
$T_i$ to $f_i$ for $i=1\ldots m$ and sending $S_i $ to $f_{m+i}$ for $i=1\ldots n$. 
It factorizes through $\B$ and gives the desired morphism of semi-affinoid $k$-algebras $\B \to \A$.
\end{proof}
 
\begin{example}
\label{rem counter example fully faithfull}
The functor from semi-affinoid $k$-algebras to $k$-analytic spaces is not fully faithful. 
Indeed,  as in Example \ref{remark example nonreduced1} let us consider
$\A  =  (R[[T_1,T_2]] \otimes_R k )/ (T_2^2)$ with associated $k$-analytic space $X$.
Let $g(T_1) \in k\llbracket T_1 \rrbracket$ be a power series which converges on the open unit disc, but such that 
$g(T_1) \notin R\llbracket T_1 \rrbracket \otimes_R k$.
Then $T_2 g(T_1) \in \Gamma(X,  \Oo_{X})$ and hence it defines a morphism of $k$-analytic spaces 
$\phi : X \to \mathbb{A}^{1, {\rm an}}_k$ whose image is the origin of $\mathbb{A}^{1, {\rm an}}_k$.
In particular it also defines a morphism $X \to E$ where $E$ is the closed unit disc over $k$ which is a $k$-affinoid space (with affinoid algebra $\B = k\langle T \rangle$).
But $\phi$ is not induced by a morphism of semi-affinoid $k$-algebras $\B \to \A$.
\end{example} 
 
\begin{rem}
\label{rem:biblio}
Theorem \ref{theo:SpecialIntegral} was already stated in A. J. de Jong's article \cite[Remark 7.4.2]{deJongCryst} without proof. 
Theorem \ref{theo:bound} is also stated in \cite[Lemma 2.14]{Nic2009} but its proof is obtained as a corollary of \cite[Remark 7.4.2]{deJongCryst}.
Corollary \ref{cor:equivalence} is also stated in \cite[Remark 2.56]{KappUni} without proof.
\end{rem} 

\begin{rem}
\label{remark reduced space}
Let $\A$ be a semi-affinoid $k$-algebra with associated $k$-analytic space $X$.
Let $\Ared \coloneqq \A / ({\rm nil} \ \A)$ be the reduced ring associated to $\A$.
Then $\Ared$ is a reduced semi-affinoid $k$-algebra.
Let $Y$ be the $k$-analytic space associated with $\Ared$.
Then $\A \to \Ared$ induces a closed immersion of $k$-analytic spaces $Y \to X$ which is a bijection of sets.
Let  $B$ be a special  $R$-model of $\Ared$.
Since $B \subset B\otimes_R k \simeq \Ared$, we deduce that $B$ is reduced.
Hence since $Y \simeq \Spf(B)_\eta$, by \cite[7.2.4 (c)]{deJongCryst}, $Y$ is reduced.
Hence $Y \simeq \Xred$.
\end{rem}

\section{Analytic functions and formal functions on tubes}
\label{section:Analytic functions and formal functions on tubes}

I am very grateful to Ofer Gabber for having pointed out  the reference to Proposition 6.14.4 of \cite{GroEGAIV2} 
which greatly simplifies the  proof of the following statement.

\begin{theo}
\label{theo:Completion}
Let $X$ be a $k$-analytic space associated with a reduced semi-affinoid $k$-algebra $\A$. 
Let $I\subset \Gamma(X,\Oo_X)$ be an ideal and let $U = \{x\in X \st |f(x)| <1 \ \forall f\in I\}$. Then 
\[\Gamma(X,\Oo_X)^{\wedge I} \simeq \Gamma(U,\Oo_{X}).\]
\end{theo}

\begin{proof}
Let us set $A:= \Gamma(X,\Oo_X)$.
Thanks to Corollary \ref{cor:reduced special} (i), $A$ is a reduced special $R$-algebra.
Let us choose some functions $f_1,\ldots,f_n\in A$ such that $(f_1,\ldots,f_n)=I$.
Then 
\[A^{\wedge I} \simeq A\llbracket \rho_1,\ldots,\rho_n\rrbracket / (f_i-\rho_i)_{i=1,\ldots,n}\]
is a special $R$-algebra.
Let us set $\mathfrak{Y}= \Spf(A^{\wedge I})$. 
Then by \cite[Lemma 7.2.5]{deJongCryst} (or see also \cite[5.3.5,5.3.6]{LR_ring}), the induced morphism of $k$-analytic spaces
$\mathfrak{Y}_\eta \to X$ identifies 
$\mathfrak{Y}_\eta$ with $U$ as an analytic domain of $X$.
So $U$ is the $k$-analytic space associated to the semi-affinoid $k$-algebra $\B:= A^{\wedge I}\otimes_Rk$ and by definition, $A^{\wedge I}$ is a special $R$-model of $\B$. 
Thanks to Theorem \ref{theo:SpecialIntegral}, it only remains to show that $A^{\wedge I}$ is integrally closed in $A^{\wedge I} \otimes_Rk\simeq \B$.

Since $A$ is an excellent ring (this follows from \cite[Proposition 7]{Val75} when $\cha(k)=p>0$ and from \cite[Theorem 9]{Val76} when $\cha(k)=0$), 
the morphism $\Spec(A^{\wedge I}) \to \Spec(A)$ is regular thanks to \cite[Scholie 7.8.3 (v)]{GroEGAIV2}, 
so in particular, $\Spec(A^{\wedge I}) \to \Spec(A)$ is a normal morphism 
(we refer to  \cite[D\'{e}finition 6.8.1]{GroEGAIV2} for the definitions of normal and regular morphisms of schemes). \par 
Since $A= \Gamma(X,\Oo_X)$, it follows that $A$ is integrally closed in $A\otimes_Rk$. 
So thanks to \cite[Proposition 6.14.4]{GroEGAIV2}, 
$A^{\wedge I}$ is integrally closed in $A^{\wedge I}\otimes_Rk$.
Finally, thanks to Theorem \ref{theo:SpecialIntegral}
\[A^{\wedge I} \simeq \Gamma(U,\Oo_{X}).\]
\end{proof}

\section{Reduction and connectedness}
\label{section:ReductionConnectedness}
\subsection{Reduction}
Let $\A$ be a  semi-affinoid $k$-algebra and let $X$ be its associated $k$-analytic space.  We set 
\[
\Ao  =  \{f \in \A \st |f|_{\sup} \leq 1\}, \hspace{10pt} 
\Ah = \{f \in \A \st \forall x\in X  \  |f(x)| <1 \} , 
\]
\[
\Aoo = \{f \in \A \st |f|_{\sup} < 1\}, \hspace{10pt}
\At = \Ao / \Ah, \hspace{10pt}
\Atp =  \Ao / \Aoo.
\]
When $\A$ is a $k$-affinoid algebra, 
thanks to the maximum modulus principle \cite[6.2.1.4]{BGR}, $\Ah=\Aoo$, $\At=\Atp$ 
and in this case, $\At$ corresponds to the  reduction of $\A$ as defined in \cite[section 6.3]{BGR}. 
For a general semi-affinoid $k$-algebra, 
the maximum modulus principle does not hold\footnote{Actually, $\A$ is a $k$-affinoid algebra if and only if the maximum modulus principle holds \cite[5.3.8]{LR_ring}.} 
(consider for instance $S \in R\llbracket S\rrbracket \otimes_Rk$), so in general, one has a strict inclusion $\Aoo \subset \Ah$ and 
$\At$ is a strict quotient of $\Atp$.
If $\A= \RmnT\otimes_Rk$, then $\At = \tk[T_1,\ldots,T_m]$ and 
$\Atp = \tk[T_1,\ldots,T_m]\llbracket S_1,\ldots,S_n\rrbracket$. 

For a $k$-analytic space $X$, remind that we  defined  the subsheaf $\mathcal{O}_X^\circ \subset \mathcal{O}_X$  of analytic functions $f$ such that $|f(x)| \leq 1$ for all $x$.
Likewise, we denote by $\check{\mathcal{O}}_X \subset \mathcal{O}_X$  the subsheaf   of analytic functions $f$ such that $|f(x)| <1$ for all $x$.

\begin{lemma}
\label{lemma reduced reduction}
Let $\A$ be a semi-affinoid algebra, $\Ared$ the associated reduced semi-affinoid $k$-algebra.
Let $X$ be the $k$-analytic space associated with $\A$.
By Remark \ref{remark reduced space}, the $k$-analytic space associated with $\Ared$ can be identified with $\Xred$.
Then all the natural maps in the following commutative squares are isomorphisms of $R$-algebras
\begin{equation}
\label{square reduction}
\xymatrixcolsep{5pc}\xymatrix{
\Ao / \check{\A}  \ar[r]^{\alpha} \ar[d]_{\phi}   & (\Ared)^\circ / \check{(\Ared)}  \ar[d]^{\phi_{\rm red}} \\ 
\Gamma(X, \mathcal{O}^\circ_X) / \Gamma(X, \check{\mathcal{O}}_X) \ar[r]_{\beta} & 
\Gamma(\Xred, \mathcal{O}^\circ_\Xred) / \Gamma(\Xred, \check{\mathcal{O}}_\Xred)  
}
\end{equation}

\begin{equation}
\label{square reduction 2}
\xymatrix{
\Ao / \Aoo \ar[r] \ar[d]   & (\Ared)^\circ / (\Ared)^{\circ \circ} \ar[d] \\ 
\Gamma(X, \mathcal{O}^\circ_X) / 
\{ f\in \Gamma(X, \mathcal{O}_X)  \st |f|_{\sup} <1 \}
\ar[r]  & 
\Gamma(\Xred, \mathcal{O}^\circ_\Xred) /
\{ f\in \Gamma(\Xred, \mathcal{O}_\Xred)  \st |f|_{\sup} <1 \}
}
\end{equation}
\end{lemma}

\begin{proof}
Using Theorem \ref{theo:bound}, we get isomorphisms 
$(\Ared)^\circ \simeq \Gamma(\Xred, \mathcal{O}^\circ_\Xred)$ and 
\[\check{\Ared} \simeq \Gamma( \Xred, \check{\mathcal{O}}_\Xred)\]
 which imply that $\phi_{\rm red}$ is an isomorphism.
To prove that $\beta$ is an isomorphism, we first remark that $X$ is a quasi-Stein space in the sense of \cite[Definition 2.3]{Kie67}.
Moreover we have an isomorphism of ringed spaces $(\Xred, \mathcal{O}_\Xred) \simeq (X, \mathcal{O}_X / ({\rm rad } \ \mathcal{O}_X))$ 
where ${\rm rad} \ \mathcal{O}_X$ is the nilradical of $\mathcal{O}_X$. 
Since the latter is a coherent $\mathcal{O}_X$-ideal sheaf by \cite[Corollary 9.5.4]{BGR}, 
it follows from Theorem B for quasi-Stein spaces \cite[2.4.2]{Kie67} 
that 
$\Gamma(X, \mathcal{O}_X)  \to \Gamma(\Xred, \mathcal{O}_\Xred)$ is surjective.  
Hence $\Gamma(X, \mathcal{O}^\circ_X)  \to \Gamma(\Xred, \mathcal{O}^\circ_\Xred)$ is also surjective, and so $\beta$ is surjective.
The injectivity of $\beta$ follows easily from the fact that $X$ and $\Xred$ are in natural bijection.
So $\beta$ is an isomorphism.
Since $\Ao \to (\Ared)^\circ$ is surjective,  we deduce similarly that $\alpha$ is an isomorphism.
Since \eqref{square reduction} commutes, we  deduce that $\phi$ is also an isomorphism.
The proof for \eqref{square reduction 2} is analogous.
\end{proof}

\begin{lemma}
\label{lemma red fin gen}
Let $\A$ be a semi-affinoid $k$-algebra.
Then $\At$ is a reduced $\widetilde{k}$-algebra of finite type and $\Atp$ is a reduced $\widetilde{k}$-algebra which 
is a quotient of some $\tk[T_1,\ldots,T_m]\llbracket S_1,\ldots,S_n\rrbracket$.
\end{lemma}

\begin{proof}
By Lemma \ref{lemma reduced reduction}, $\At \simeq \widetilde{\Ared}$, so we can replace $\A$ by $\Ared$ and assume that $\A$ is reduced. 
By \cite[Corollary 2.11]{KappUni},  under this assumption 
 $\Ao$ is a special $R$-algebra, so there is an isomorphism 
\[ \Ao \simeq \RmnT/I\]
for some ideal $I$ of $\RmnT$.  
The ideal $\Ah$ then contains the ideal generated by the image of $(\pi,S_1,\ldots,S_n)$ modulo $I$.
Hence $\At$ is a quotient of $\widetilde{k}[T_1,\ldots,T_m]$ proving the first part of the lemma.
Likewise $\Aoo$ contains the ideal $(\pi)$, hence we get a surjective map 
\[{\tk[T_1,\ldots,T_m]\llbracket S_1,\ldots,S_n\rrbracket} \to \Atp.\]
\end{proof}

\begin{rem}
\label{remark biggest ideal of definition}
If $\A$ is a reduced semi-affinoid $k$-algebra, 
then  $\Ah$ is the biggest ideal of definition of the special $R$-algebra $\Ao$ (see \cite[Remark 2.8]{KappUni}).
Likewise, one can prove easily that $\Aoo = {\rm rad} \ (\pi)$.
\end{rem}

\begin{art}
\label{specialization}
Let $A$ be an arbitrary special $R$-algebra and let $\Xf \coloneqq \Spf(A)$.
We denote by $\Xf_s \coloneqq \Spec(A/J)$ where $J$ is the biggest ideal of definition of $A$.
Then  there is a \emph{specialization map}  $\spe_{\Xf} : \Xf_\eta \to \Xf_s$ 
which is defined in \cite[7.1.10]{deJongCryst} on the subset of rigid points $\Xf^{\rm rig}$ and in general in \cite[$\S$ 1 p. 371]{Berko96} 
(beware that in  \cite{Berko96} the map $\spe_\Xf$ is called the reduction map).
\end{art}

\begin{defi}
\label{defi:reduction}
Let $X$ be a semi-affinoid $k$-analytic space coming from a semi-affinoid $k$-algebra $\A$. 
We set $\widetilde{X}= \Spec(\At)$.
We call $\tX$ the canonical reduction of $X$. 
According to Lemma \ref{lemma red fin gen}, it is a reduced $\tk$-scheme of finite type. 
If  $Y$ is a semi-affinoid $k$-analytic spaces and $\varphi : X \to Y$ is a morphism of $k$-analytic spaces, then we get an associated morphism 
$\varphi^* : \Gamma(Y, \mathcal{O}^\circ_Y) \to \Gamma(X, \mathcal{O}^\circ_X)$. 
Hence by Lemma \ref{lemma reduced reduction}, we can functorially associate a morphism  $\widetilde{\varphi} : \widetilde{X} \to \widetilde{Y}$.
If $\varphi$ comes from a morphism of semi-affinoid $k$-algebras $\psi : \B \to \A$, then  $\widetilde{\varphi}$ is induced by 
$\widetilde{\psi} : \widetilde{\B} \to \widetilde{\A}$.
\end{defi}

\begin{art}
\label{reduction map}
As in \cite[section 2.4]{Berko90} one can define a canonical reduction map $\red : X \to \widetilde{X}$ in the following way. 
If $x\in X$ we obtain an associated map $\chi_x : \A \to \H(x)$, which gives rise to a map 
$\widetilde{\chi_x}: \At \to \widetilde{\H(x)}$. 
Then 
\[
\begin{array}{cccc}
\red : & X & \to & \tX \\
& x & \mapsto & \Ker(\widetilde{\chi_x})
\end{array}
\]
To stress the dependence in $X$, we might write $\red_X$ instead of $\red$.
Similarly as in \cite[Corollary 2.4.3]{Berko90} one can check that the canonical reduction map $\red$ is anticontinuous, i.e. the inverse image of an open set is closed.
\end{art}

\begin{rem}
\label{remark reduction reduce commutes}
Identifying $\Xred$ with the $k$-analytic space associated with $\Ared$ (see Remark \ref{remark reduced space}), 
and using Lemma  \ref{lemma reduced reduction}, we get a commutative diagram 
\[
\xymatrix{
\Xred \ar[r] \ar[d]_{{\red}_{\Xred}} & X \ar[d]^{{\red}_X} \\
\widetilde{\Xred} \ar@{=}[r] & \tX
}
\]
\end{rem}

\begin{rem}
\label{remark reduction formal}
Let $A$ be a special $R$-model of the semi-affinoid $k$-algebra $\A$, with associated $k$-analytic space $X$. 
Let $J$ be the biggest ideal of definition $A$.
The injection $A \to \Ao$ induces an injective morphism of $\tk$-algebras $\varphi \colon A/J \to \At$ which induces a morphism $\iota \colon \tX \to \Xf_s$.
If $\Xf \coloneqq \Spf(A)$ we get a commutative diagram 
\[
\xymatrix{
      &X \simeq  \Xf_\eta \ar[ld]_{ {\rm red}} \ar[rd]^{\spe_\Xf}  & \\
      \tX \ar[rr]_\iota & & \Xf_s
}
\]
If $\A$ is  reduced, then $A \coloneqq \Ao$ is a special $R$-model and in that case $J = \Ah$ (see Remark \ref{remark biggest ideal of definition}), 
hence in that case, $\varphi$ and $\iota$ are isomorphisms. 
In general, since $\At$ is a finitely generated $\tk$-algebra, we can find  $f_1,\ldots,f_m \in \Ao$ such that 
$\widetilde{f_1}, \ldots, \widetilde{f_m}$ generate the $\tk$-algebra $\At$.
By \cite[Corollary 2.12]{KappUni}, $A' \coloneqq A[f_1,\ldots,f_m] \subset \Ao$ is still a special $R$-model of $\A$, and by construction, if $J'$ is the biggest ideal of definition of $A'$, 
the map  $\varphi : A'/J' \to \At$ is surjective.
Since $\varphi$ is injective anyway, it is an isomorphism.
In conclusion, if $X$ is semi-affinoid $k$-analytic space, one can always find a special $R$-scheme $\Xf = \Spf(A)$ which is a model of $X$ 
such that $\tX \simeq \Xf_s$ and such that ${\rm red}$ can be identified with $\spe_\Xf$.
\end{rem}

\begin{cor}
\label{cor:Tubes1}
Let $X$ be a  semi-affinoid $k$-analytic space.
Let $Z$ be a Zariski-closed subset of $\tX$ and let $Y:= \red^{-1}(Z)$. 
Then  $Y$ is a  semi-affinoid $k$-analytic space and the naturally induced map $\widetilde{Y} \to \tX$ is a closed immersion with image $Z$.
\end{cor}

\begin{proof}
Let $\A$ be the semi-affinoid $k$-algebra of $X$.
Using Remark \ref{remark reduction reduce commutes}, we can replace $\A$ by $\Ared$ and assume that $\A$ is reduced.
Under this assumption, by \cite[Corollary 2.11]{KappUni}, $\Ao$ is a special $R$-algebra, and 
$X \simeq \Spf(\Ao)_\eta$.
Let $I$ be an ideal of $\Ao$ such that $Z = V(\widetilde{I})$ where $\widetilde{I}=\{\widetilde{f} \st f\in I\}$. 
Then as already seen in the proof of Theorem \ref{theo:Completion},  $Y \simeq \Spf((\Ao)^{\wedge I})_\eta$, hence $Y$ is the $k$-analytic space associated to the semi-affinoid $k$-algebra 
$\B \coloneqq (\Ao)^{\wedge I} \otimes_R k$.
Then according to Theorem \ref{theo:Completion}, 
$\Bo = (\Ao)^{\wedge I}$. 
Since $\Ah$ is an ideal of definition of $\Ao$ \cite[2.8]{KappUni}, 
$\Ah + I$ is also an ideal of definition of $\Bo$.
Since $\Bh$ is the biggest ideal of definition of $\Bo$ \cite[2.8]{KappUni}, it follows that 
\[\Bt = \Bo / \Bh =  \left((\Ao)^{\wedge I} / ( \Ah + I)\right)_{\red} = (\At / I)_{\red}.\] 
\end{proof}

\begin{rem}
\label{rem:reduction compatibility}
In general, when $U$ is an affinoid domain of the $k$-affinoid space $X$, it is difficult to describe the induced canonical reduction map 
$\widetilde{U}\to \widetilde{X}$.
However, it is proved in \cite[7.2.6.3]{BGR} that when $U$ is the tube of a principal open subset of $\tX$ (i.e. when $U = \red^{-1}(D(\widetilde{f}))$ for some $f\in \Ao$), then 
$\widetilde{U}\to \widetilde{X}$ is an open immersion with image $D(\widetilde{f})$. 
The above Corollary is the counterpart of \cite[7.2.6.3]{BGR} for Zariski-closed subsets of the canonical  reduction.
\end{rem}

\begin{lemma}
\label{lemma:SurjectivityReduction}
Let $A$ be an $R$-flat special $R$-algebra and let $\Xf = \Spf(A)$.
Then the specialization  map $\spe_\Xf : \Xf_\eta \to \Xf_s$ is surjective. 
\end{lemma}

\begin{proof}
This proof is strongly inspired by the proof of Lemme 1.2 in \cite{Poi}.
Let $\widetilde{x} \in \Xf_s$, and let $r$ be the transcendence degree of the field extension 
$\widetilde{k}(\widetilde{x}) / \widetilde{k}$.
Let $K$ be the completion of $k(U_1,\ldots,U_r)$ with respect to the Gauss norm.
Then $K$ is a discretely valued non-archimedean field extension of $k$, and $\widetilde{K} \simeq \widetilde{k}(U_1,\ldots,U_r)$.
Hence $\widetilde{x}$ is the image of a closed point $\widetilde{y}\in \Xf_s \times_{\Spec(\widetilde{k})} \Spec(\widetilde{K})$ 
with respect to the canonical map $\Xf_s \times_{\Spec(\widetilde{k})} \Spec(\widetilde{K}) \to \Xf_s$.

Using base change extension induced by the inclusion $R \to K^\circ$ (see \cite[7.2.6]{deJongCryst}), 
 \[\Xf' \coloneqq \Xf \times_{\Spf(R)} \Spf(K^\circ)\]
is a special formal scheme over $K^\circ$ of the form $\Xf' = \Spf(A')$ where $A' \coloneqq A \widehat{\otimes}_R K^\circ$.
Let $J$ be the biggest ideal of definition of $A$.
Then $JA'$ is an ideal of definition of $A'$ and $J' \coloneqq {\rm rad} ( J A')$ is the biggest ideal of definition of $A'$.
There is a well-defined morphism of $\tk$-algebras 
$\alpha \colon A/J \otimes_{\tk} \widetilde{K} \to A' / J'$ defined by 
$\widetilde{a}\otimes \widetilde{\lambda} \mapsto \widetilde{a\otimes \lambda}$ 
for $a\in A$, $\lambda \in K^\circ$ and where $\widetilde{\cdot}$ stands for the various residue maps.
We claim that $\alpha$ induces an isomorphism 
\begin{equation}
\label{eq:reduction extension}
( A/J \otimes_{\tk} \widetilde{K})_{\rm red} \simeq A' / J'.
\end{equation}
Indeed by construction $\alpha$ is surjective, so it remains to prove that $\ker \alpha$ is the nilradical of $A/J \otimes_{\tk} \widetilde{K}$.
There is also a well-defined surjective map $\beta \colon A' \to  A/J \otimes_{\tk} \widetilde{K}$ defined by $a \otimes \lambda \mapsto \widetilde{a} \otimes \widetilde{\lambda}$, 
and by construction $\alpha \circ \beta \colon A' \to A' /J'$ is the quotient map. 
So let $\widetilde{z} = \sum_i \widetilde{a_i} \otimes \widetilde{\lambda_i} \in \ker \alpha$ for some $a_i \in A$ and $\lambda_i \in K^\circ$ 
and let  $z \coloneqq \sum_i a_i \otimes \lambda_i \in A'$. 
Then $\beta(z) = \widetilde{z}$, hence $z\in \ker \alpha \circ \beta$ hence $z \in J'$. 
Since $J' \coloneqq {\rm rad} ( J A')$ we have $z^n \in J A'$ for some $n\in \N^*$.
Hence $\beta(z^n)=0$, hence $\widetilde{z}^n=0$, thus proving \eqref{eq:reduction extension}.
We get a natural composite morphism 
\[
\iota \colon (\Xf')_s \xrightarrow[]{\iota_1} \Xf_s \times_{\Spec(\widetilde{k})} \Spec(\widetilde{K}) \xrightarrow[]{\iota_2} \Xf_s
\]
where $\iota_1$ is induced by \eqref{eq:reduction extension} and hence is bijective, and $\iota_2$ is the canonical map.
%and we get  a canonical isomorphism 
%$\Xf_s \times_{\Spec(\widetilde{k})} \Spec(\widetilde{K}) \simeq  (\Xf')_s$.
Hence we can identify $\widetilde{y}$ with a closed point of $(\Xf')_s$.
We also get a commutative diagram of sets
\[\xymatrix{
(\Xf')_\eta   \ar[r]^\alpha \ar[d]_{\spe_{\Xf'}} & (\Xf)_\eta \ar[d]^{\spe_\Xf} \\
(\Xf')_s    \ar[r]_{\iota}& \Xf_s
}
\]
Thanks to \cite[Lemma 2.3, Remark 2.5]{KappUni} we know that $\spe_{\Xf'}$ induces a surjective map from the set of rigid points 
$(\Xf')^{\rm rig}$   to the set of closed points of $(\Xf ')_s$.  
Hence we can find $y \in (\Xf ')_\eta$ such that $\spe_{\Xf '}(y) = \widetilde{y}$. 
In conclusion, if $x \coloneqq \alpha(y)$ we get $\spe_\Xf(x) = \widetilde{x}$, proving the surjectivity of $\spe_{\Xf}$.
\end{proof}

\begin{cor}
\label{corollary surjectivity reduction}
Let $X$ be a semi-affinoid $k$-analytic space.
Then the canonical reduction map $\red : X \to \widetilde{X}$ is surjective.
\end{cor}

\begin{proof}
This follows from Lemma \ref{lemma:SurjectivityReduction} and Remark \ref{remark reduction formal}.
\end{proof}

\subsection{Connected components}
In this subsection, we consider a  semi-affinoid $k$-algebra $\A$ with associated $k$-analytic space $X$.
It follows from Theorem \ref{theo:bound} and Remark \ref{remark reduced space} that $\Spec(\A)$ is connected if and only if $X$ is connected.
Indeed by Remark \ref{remark reduced space}, we can assume that $\A$ is reduced, and we conclude since 
the connected components of $X$ are in correspondence with the set of idempotents of 
$\Gamma(X,\O_X)$ which themselves are equal to the set of idempotents of $\{ f\in \Gamma(X, \O_X) \st \ |f| < \infty\} =\A$.
From the Noetherianity of $\A$ it follows that $\A$ can be uniquely decomposed as 
\begin{equation}
\label{eq:dec1}
\A \simeq \A_1 \times \cdots \times \A_n
\end{equation} 
where each $\A_i$ is a semi-affinoid $k$-algebra such that $\Spec(\A_i)$ is connected. 
If we denote by $X_i$ the $k$-analytic space associated to $\A_i$ it follows that 
\begin{equation}
\label{eq:dec2}
X= X_1\coprod \cdots \coprod X_n
\end{equation} 
is the decomposition of $X$ in connected components. 
These remarks easily imply the following.

\begin{lemma}
\label{lemma:ReductionConnected}
A semi-affinoid $k$-analytic space $X$ is connected if and only if $\tX$ is connected. 
\end{lemma} 

\begin{proof}
The equations \eqref{eq:dec1} and \eqref{eq:dec2} imply that $\tX \simeq \coprod \tX_i$. 
So if $X$ is not connected, $\tX$ is also not connected. 
Conversely, since $\red : X \to \tX$ is surjective (Corollary \ref{corollary surjectivity reduction})  and anticontinuous, if one has a decomposition 
$\tX = U_1 \coprod U_2$ in two nonempty closed-open sets, then $X = {\rm red}^{-1}(U_1) \coprod  \red^{-1}(U_2)$ 
is a decomposition of $X$ in nonempty closed-open sets. 
\end{proof}

\begin{cor}
\label{cor:Tubes2}
Let $X$ be a semi-affinoid $k$-analytic space.
Let $Z$ be a Zariski-closed subset of $\tX$ and let $Y:= \red^{-1}(Z)$. 
Then $Y$ is connected if and only if $Z$ is connected.
\end{cor}

\begin{proof}
This follows from Corollary \ref{cor:Tubes1} and from Lemma \ref{lemma:ReductionConnected}.
\end{proof}

Using Theorem \ref{theo:SpecialIntegral}, one checks that if $A$ is a reduced special $R$-algebra which is integrally closed in $A\otimes_Rk$, 
then there is a one-to-one correspondence between the connected components of $\Spec(A)$ and the connected components of $\Spf(A)_\eta$.

\begin{example}
In general, if $A$ is a reduced special $R$-algebra with associated $k$-analytic space $X$, the connected components of $\Spec(A)$ do not coincide with the connected components of $X$ 
as the example $A = \Zp\langle T\rangle /(T^2 + pT)$ shows.
\end{example}

\section{Analytic functions and formal functions on tubes: a graded version}
\label{section: Graded Version}
In this section, we fix $\A$ a reduced semi-affinoid $k$-algebra.
Let us recall (see \cite[2.11]{KappUni}) that this implies that $\Ao$ is a special $R$-algebra.
\begin{art}
\label{notations graded}
For $r\in \R_+^*$ we set 
\[
\Aor = \{ f \in \A \st |f|_{\sup} \leq r \}, \hspace{10pt}
\Aoor  = \{ f \in \A \st |f|_{\sup}< r \}, \hspace{10pt}
\Atr  = \Aor / \Aoor.
\]
By definition, $\Ao_1 = \Ao$, $\Aoo_1 = \Aoo$ and $\widetilde{\A}_1^+ = \Atp$.
We also denote by 
\[\rho(\A) = \{ |f|_{\sup} \st f\in \A, \ f\neq 0\}\subset \R_+^*.\]
By \cite[1.2.5.9]{KappThesis}, $\rho(\A) \subset \val$. 
We denote by $G$ the subgroup of $\rk$ generated by $\rho(\A)$. 
\end{art}

\begin{lemma}
\label{lemma r finitely generated}
Let $r\in \R_+^*$. Then $\Aor$, $\Aoor$ and $\Atr$ are finitely generated 
$\Ao$-modules.
\end{lemma}
\begin{proof}
Let us pick some $\lambda \in k^*$ such that $|\lambda|r \leq 1$. Then 
\[
\begin{array}{rcl}
\Aor & \to & {\Ao}_{|\lambda| r} \\
f & \mapsto & \lambda f
\end{array}
\]
is an isomorphism of $\Ao$-modules.
So replacing $r$ by $|\lambda|r$, we can assume that $r\leq 1$.
Then $\Aor$ is an ideal of $\Ao$, hence it is a finitely generated $\Ao$-module because $\Ao$ is Noetherian.
We conclude since $\Aoor$ (resp. $\Atpr$) is a submodule (resp. a quotient) of $\Aor$.
\end{proof}

\begin{lemma}
\label{lemma:discrete}
The index  $[G:\ |k^*|]$ is finite.
As a consequence $G\simeq \Z$.
\end{lemma}

\begin{proof}
Let $\kappa := \sup \{ \lambda \st \lambda <1 \ \text{and}  \ \lambda \in \rho(\A) \}.$
We claim that $\kappa$ is actually a maximum, that is to say, there exists $\lambda<1$ with $\lambda \in \rho(\A)$ such that $\kappa = \lambda$ 
(this implies in particular that  $\kappa <1$). 
Indeed, if $\kappa$ was not a maximum, we could find an increasing sequence $(\lambda_n)_{n\in \N}$ with $\lambda_n <1$ in $\rho(\A)$. 
Then $I_n = \{f\in \A \st |f|_{\sup} \leq \lambda_n \}$ 
would be an infinite increasing sequence of ideals of $\Ao$, which is Noetherian. 
This would be a contradiction. 
Hence, $\kappa <1$ and $\kappa \in \rho(\A)$.

Since $\kappa \in \rho(\A)$, we can write $\kappa = |\pi|^{\frac{a}{b}}$ with $a$ and $b$ relatively prime, and $b>0$.
One can prove easily using Bezout's Theorem that $a=1$.
Hence $\kappa = |\pi|^{\frac{1}{b}}$. 
Likewise, using again Bezout's Theorem, one can prove that if $ |\pi|^{\frac{\alpha}{\beta}} \in \rho(\A)$ with $\beta>0$ and $\alpha$ and $\beta$ 
coprime, then $\beta \leq b$.
Hence if  
$v :={\rm lcm}(1,2,\ldots,b)$, one has $\rho(\A) \subset |\pi|^{\frac{1}{v}\Z}$.
\end{proof}

\begin{rem}
If $\| \cdot \|$ is a $k$-Banach algebra norm on $\A$, according to \cite[1.2.5.8]{KappThesis}, for any $f\in \A$,
$|f|_{\sup} = \lim_{n\in \N} \sqrt[n]{\|f^n\|}$.
The notation $\rho(\A)$ that we have introduced is compatible with \cite[1.3]{Berko90} where the spectral radius of an element of a Banach algebra $\A$ is defined to be 
$\rho(f) = \lim_{n\in \N} \sqrt[n]{\|f^n\|}$. 
Let us recall (see \cite[1.2.5.4]{KappThesis}) that  a surjective morphism  of semi-affinoid $k$-algebras 
\[\RmnT \otimes_Rk \to \A\] 
induces a $k$-Banach norm on $\A$ by taking the residue semi-norm of the Gauss norm on 
\[\RmnT \otimes_Rk.\]
\end{rem}

\begin{example}
In general, $\rho(\A)$ is not a subgroup, and not even a monoid. 
For instance if 
$\A = \Qp(\sqrt{p}) \times \Qp(\sqrt[3]{p})$ then 
$\rho(\A) = p^{\frac{1}{2}\Z} \cup p^{ \frac{1}{3} \Z } $ which is not a monoid. 
\end{example}
The following definition is inspired by \cite[section 3]{Temloc2}.

\begin{defi}
We define the \emph{total reduction ring of $\A$} as 
\[ \widetilde{\A}^+_{\text{tot}} := \bigoplus_{r \in G} \Atr. \]
\end{defi}

\begin{cor}
\label{cor:reduced}
The total reduction ring $\Atot$ of $\A$ is excellent and reduced.
\end{cor}

\begin{proof}
By Lemma \ref{lemma:discrete} $[G:|k^*|]$  is finite, so we can introduce 
$n \coloneqq [G:|k^*|] \in \N$. 
Let $r_1,\ldots, r_n\in G$ be some representatives of the classes of $G/ |k^*|$.
For each $i\in \{1,\ldots,n\}$, 
by Lemma \ref{lemma r finitely generated}, we can fin a finite family 
$(f_{i,j})_{j\in J}$ which is a finite set of generators of the $\Ao$-module $\widetilde{\A}^+_{r_i}$. 
Let us denote by $\widetilde{\pi}$ (resp. $\widetilde{\left( \frac{1}{\pi} \right)}$) the image of $\pi$ in $\widetilde{\A}^+_{|\pi|}$ 
(resp of $\frac{1}{\pi}$ in $\widetilde{\A}^+_{\frac{1}{|\pi|}}$).
We get that 
\[\bigoplus_{r \in G} \Atr = \Atp \left[ \widetilde{\pi}, \widetilde{\left( \frac{1}{\pi} \right)}, f_{i,j} \right]. \]
Hence $\Atot$ is a finitely generated $\Atp$-algebra. 
Since $\Atp$ is excellent, it follows that $\Atot$ is also excellent. 
For the reduceness, let us consider a nonzero element $\widetilde{f} \in \Atr$ for  some $f\in \A$ with $|f|_{\sup}=r$. 
Then for any integer $k>0$, the associated element $\widetilde{f}^k = \widetilde{f^k}\in \widetilde{\A}^+_{r^k}$ is also nonzero since $|f^k|_{\sup} = |f|_{\sup}^k = r^k$.
\end{proof}

\begin{rem}
\label{rem:Aoor}
Let $r\in \rho(\A)$. 
Since $\rho(\A)$ is discrete in $\R^*_+$ (see Lemma \ref{lemma:discrete}), there exists a real number $s\in \rho(\A)$ which is 
the biggest element of $\rho(\A)$ such that $s<r$. 
It follows that $\Aoor = \Aos$.
\end{rem}

We fix $I$ an ideal of $\Ao$.
If $M$ is an $\Ao$-module, we denote by 
$M^{\wedge I}$ the completion of $M$ with respect to the $I$-adic topology. 
So 
\[ M^{\wedge I} \simeq \varprojlim_{n} M/I^nM.\]

\begin{lemma}
\label{lemma:ses}
Let $r\in \rho(\A)$. 
There is a short exact sequence of $\Ao$-modules
\begin{equation}
\label{eq:ses}
0 \rightarrow (\Aoor)^{\wedge I} \rightarrow (\Aor)^{\wedge I} \rightarrow (\Atr)^{\wedge I} \rightarrow 0.
\end{equation}
\end{lemma}

\begin{proof}
By definition of $\Atr$, there is a short exact sequence 
\begin{equation}
\label{eq:ses2}
0 \rightarrow \Aoor \rightarrow \Aor \rightarrow \Atr \rightarrow 0
\end{equation} 
of finitely generated $\Ao$-modules. 
So the $I$-adic completion of \eqref{eq:ses2} remains exact (see \cite[8.7 and 8.8]{Mat89}).
\end{proof}

\begin{art}
\label{notations A B}
Let $I$ be an ideal of $\Ao$ and let us set $U= \{x\in X \st |f(x)|<1 \ \forall f\in I\}.$ 
Let us denote by $\B$ its associated semi-affinoid $k$-algebra, which can be defined as 
\[ \B = (\Ao\llbracket S_1,\ldots,S_n\rrbracket / (f_i-S_i)_{i=1\ldots n})\otimes_Rk \]
where $I=(f_1,\ldots,f_n)$.
According to Theorem \ref{theo:Completion}, one has $\Bo\simeq (\Ao)^{\wedge I}$.
We denote by $Y$ the $k$-analytic space associated with $\B$.
If $g\in \B$ we denote by  $|g|_{\sup} = \sup_{y\in Y} |g(y)|$.
\end{art}

\begin{rem}
\label{rem:Bor}
Let $r \in |k^*|$ and let $\lambda \in k$ with $|\lambda| = r^{-1}$.
Multiplication by $\lambda$ induces an isomorphism of $\Ao$-modules 
$\Ao_r \xrightarrow[]{ \times \lambda} \Ao$. 
Completing with respect to $I$ one gets an isomorphism of $(\Ao)^{\wedge I}$-modules 
$ (\Aor)^{\wedge I}  \simeq (\Ao)^{\wedge I}$.
Finally using Theorem \ref{theo:Completion} we get an isomorphism of $(\Ao)^{\wedge I}$-modules 
$\Bor \simeq (\Aor)^{\wedge I}$ obtained as the composition 
\[ (\Aor)^{\wedge I}  \xrightarrow[]{\times \lambda} (\Ao)^{\wedge I} \rightarrow 
\Bo \xrightarrow[]{\times \lambda^{-1}} \Bor.\]
More generally, if $r\in \rho(\A)$, we can find $s\in |k^*|$ with $r \leq s$, leading to an inclusion of 
$\Ao$-modules $\Aor \to \Aos$.
Then completing with respect to $I$ we get an inclusion of $(\Ao)^{\wedge I}$-modules 
$(\Aor)^{\wedge I} \to (\Aos)^{\wedge I}$.
Using the above identifications, we can assimilate $(\Aor)^{\wedge I}$ as a $\Bo$-submodule of $\Bo_s$, hence as a 
$\Bo$-submodule of $\B$.
\end{rem}

\begin{lemma}
\label{lemma:Boo}
Let $r\in |k^*|$.
\begin{enumerate}[(i)]
\item Let $g\in \Bor \simeq  (\Aor)^{\wedge I}$ (see Remark \ref{rem:Bor}).
Let $\widetilde{g} \in (\Atp_r)^{\wedge I}$ be the image of $g$ by the reduction map of the short exact sequence 
\eqref{eq:ses} of Lemma \ref{lemma:ses}.
Then $\widetilde{g}=0$ if and only if $\supn{g} <r$. 
Equivalently, $\widetilde{g}\neq 0$ if and only if $\supn{g} = r$.
\item There is a natural isomorphism $\Boo_r  \simeq  (\Aoo_r)^{\wedge I}$.
\end{enumerate}
\end{lemma}

\begin{proof}
Using the same arguments as in Remark \ref{rem:Bor}, we can assume that  $r=1$. 
We then consider the short exact sequence of Lemma \ref{lemma:ses} for $r=1$:
\[0 \rightarrow (\Aoo)^{\wedge I} \rightarrow (\Ao)^{\wedge I} \rightarrow (\Atp)^{\wedge I} \rightarrow 0.\]
Let us then consider $g\in \Bo \simeq  (\Ao)^{\wedge I}$ and let us assume that $\widetilde{g}=0$.
This implies that $g\in (\Aoo)^{\wedge I}$.
By Remark \ref{rem:Aoor}, there exists $s<1$ such that $(\Aoo)^{\wedge I} = (\Aos)^{\wedge I}$ for some $s<1$. 
It follows that $\supn{g} \leq s <1$. 
Conversely, let us assume that $\supn{g}<1$. 
There exists an integer $d\in \N$  such that $\supn{g^d} \leq |\pi|$.    
 Hence, according to Remark \ref{rem:Bor}, $g^d  \in \Bo_{|\pi|} \simeq  (A^\circ_{|\pi|})^{\wedge I} \subset (\Aoo)^{\wedge I}$.
 This proves (i), and (ii) follows from (i).
\end{proof}

We can now generalize Theorem \ref{theo:Completion} to an arbitrary $r\in \rho(\A)$. 

\begin{prop}
\label{prop:Bor}
We use the notations of \ref{notations A B}.
\begin{enumerate}[(i)]
\item There is an inclusion $\rho(\B) \subset   \rho(\A)$.
\item Let $r\in \rho(\A)$. 
There are isomorphisms 
$\Bor \simeq (\Aor)^{\wedge I}$, 
$\B^{\circ\circ}_r \simeq (\Aoor)^{\wedge I}$ and 
$\Btr \simeq (\Atr)^{\wedge I}$.
\item There is a natural isomorphism
$ \Btot \simeq \bigoplus_{r\in G} (\Atr)^{\wedge I}$.
\end{enumerate}
\end{prop}

\begin{proof}
Let $g\in \B$ be a nonzero element.
Then there exists $s\in \rho(\A)$ such that $g\in \Bo_s$.
Since $\rho(\A)$ is  discrete by Lemma \ref{lemma:discrete}, we can then define the smallest element $r\in \rho(\A)$ such that $g\in \Bor$.
By Remark \ref{rem:Bor}, for each $r\in \rho(\A)$, we can naturally identify $(\Aor)^{\wedge I}$ with a $\Bo$-submodule of $\B$ and under these identifications, 
 $\B = \cup_{r\in \rho(\A)} (\Aor)^{\wedge I}$.
Since $\rho(\A)$ is  discrete, we can then define the smallest element $r\in \rho(\A)$ such that $g\in (\Aor)^{\wedge I}$.
We then consider the short exact sequence \eqref{eq:ses}
\[0 \rightarrow (\Aoor)^{\wedge I} \rightarrow (\Aor)^{\wedge I} \rightarrow (\Atr)^{\wedge I} \rightarrow 0.\]
The minimality of $r$ and  Remark \ref{rem:Aoor} imply that $g\notin  (\Aoor)^{\wedge I}$.
It follows that $\widetilde{g}\neq0$ where $\widetilde{g}\in (\Atr)^{\wedge I}$ denotes the reduction of $g$ in $(\Atr)^{\wedge I}$.
Thanks to Lemma \ref{lemma:discrete} we can pick some $d\in \N^*$ such that $r^d \in |k^*|$.
Then $g^d\in (\A^\circ_{r^d})^{\wedge I}$.
Since $(\Atot)$ is reduced and excellent (Corollary \ref{cor:reduced}), it follows that   $(\Atot)^{\wedge I}$ is also reduced.
Hence 
$\widetilde{g^d}=\widetilde{g}^d \neq0$ 
in $\left(\widetilde{\A}^+_{r^d}\right)^{\wedge I}$.
Since $r^d \in |k^*|$, Lemma \ref{lemma:Boo} implies that $\supn{g^d} = r^d$.
So  $\supn{g}=r$. This proves (i) as well as  (ii) and (iii).
\end{proof}

We also obtain the following generalization of \cite[Lemma 2.1]{bosch1985stable}.

\begin{cor}
\label{cor:intrinsec}
Let $X$ be a $k$-affinoid space, and let $Z\subset \tX$ be a Zariski-closed subset. 
Then $\widehat{\tX_{/Z}}$, the formal completion of $\tX$ along $Z$, depends intrinsically on the $k$-analytic space $\red^{-1}(Z)$.
\end{cor}
\begin{proof}
Let us denote by $\B$ the semi-affinoid $k$-algebra of $\red^{-1}(Z)$. 
The space $\widehat{\tX_{/Z}}$ is the formal scheme associated to the adic $\tk$-algebra $(\Atp)^{\wedge I}$.
Thanks to the short exact sequence \eqref{eq:ses} for $r=1$,
\[  (\Atp)^{\wedge I} \simeq {\Ao}^{\wedge I} / {\Aoo}^{\wedge I}.\]
Thanks to Theorem \ref{theo:Completion} and Lemma \ref{lemma:Boo} (ii), one gets that 
\[  (\Atp)^{\wedge I} \simeq \Bo/\Boo = \Btp\]
which depends intrinsically on $\red^{-1}(Z)$ since $\B$ depends intrinsically on $\red^{-1}(Z)$.
\end{proof}

\begin{rem}
\label{remark reduction formal }
More generally, if $X$ is the $k$-analytic space associated to the reduced semi-affinoid $k$-algebra $\A$, then $\Atp$ is naturally an adic algebra 
with an ideal of definition given by $\Ah$.
This adic algebra is isomorphic to a quotient of $\tk[T_i]\llbracket S_j\rrbracket$.
Let us set $\X:= \Spf(\Atp)$.
Then $|\X| = \tX$.
Let $Z$ be a Zariski closed subset of $\tX$ and let $Y$ be the $k$-analytic space defined by 
$Y=\red^{-1}(Z)$ and let $\B$ be its associated semi-affinoid $k$-algebra.  
One shows similarly that the inclusion of analytic domain $Y \to X$ induces an isomorphism
$\widehat{\X_{/Z}} \simeq \Spf(\Btp)$ 
where we denote by $\widehat{\X_{/Z}}$ the completion of $\X$ along $Z$.
In particular, $\widehat{\X_{/Z}}$ depends intrinsically on the $k$-analytic space $\red^{-1}(Z)$.
\end{rem}

\section{Additional remarks}
\label{section:additional remarks}

\subsection{Finite morphisms}

\begin{prop}
\label{prop:finite}
Let $\varphi : \B \to \A$ be a finite morphism of semi-affinoid $k$-algebras with $\A$ reduced. 
Then $\varphi^\circ : \Bo \to \Ao$ is finite. 
\end{prop}

\begin{proof}
Since $\Bo \to (\B_{\text{red}})^\circ$ is finite, we can also assume that $\B$ is reduced. 
Let $f_1,\ldots,f_n$ be elements of $\A$ such that 
\begin{equation}
\label{eq:finite}
 \A= \B[f_1,\ldots,f_n].
\end{equation}
Each  $f\in \{f_1,\ldots,f_n\}$ satisfies a unitary polynomial equation with coefficients in $\B$ of the form 
\begin{equation}
\label{equation multiple}
f^d + b_{d-1}f^{d-1} + \ldots + b_1 f +b_0 =0.
\end{equation} 
Then for $m\in \N^*$, multiplying by $\pi^{md}$, \eqref{equation multiple} becomes 
\begin{equation*}
\label{equation multiple 2}
(\pi^mf)^d + \pi^m b_{d-1}(\pi^mf)^{d-1} + \ldots + \pi^{m(d-1)}b_1 (\pi^m f) + \pi^{md} b_0 =0.
\end{equation*} 
But for $m$ big enough, all the coefficients $\pi^m b_{d-1}, \ldots, \pi^{m(d-1)}b_1, \pi^{md} b_0$ appearing in the above equation belong to $\Bo$. 
Hence,  for $m$ big enough, $\pi^mf_i$ satisfies a unitary polynomial  equation with coefficients in $\Bo$.
So replacing each $f_i$ by $\pi^mf_i$ (which will not change \eqref{eq:finite}), we can assume that the $f_i$'s belong to $\Ao$ and are integral over $\B^\circ$.
So, thanks to \eqref{eq:finite}, $\Bo[f_1,\ldots,f_n]$ is a special $R$-model of $\A$.
According to \cite[2.10]{KappUni}, $\Ao$ is finite over $\Bo[f_1,\ldots,f_n]$, and hence also over $\Bo$. 
\end{proof}

We conjecture that  for an arbitrary  non-archimedean non-trivially valued field $k$, a similar statement holds for quasi-affinoid $k$-algebras 
(see \cite[2.1.8]{LR_ring} for the definition of a quasi-affinoid $k$-algebra).

\begin{cor}
\label{cor:finite}
Let $\varphi : \A \to \B$ be a finite morphism of semi-affinoid $k$-algebras. 
The associated morphisms
$\widetilde{\varphi} : \At \to \Bt$, 
$\widetilde{\varphi}^+ : \Atp \to \Bt^+$, 
$\widetilde{\varphi}_{\text{tot}} : \Atot \to \widetilde{\B}^+_{\text{tot}}$ are finite.
\end{cor}
\begin{proof}
Since the above morphisms do not change if one replaces $\A$ and $\B$ by $\A_{\red}$ and $\B_{\red}$, we can assume that $\A$ and $\B$ are reduced. 
The first two points then follow from Proposition \ref{prop:finite}. 
To prove that $\widetilde{\varphi}_{\text{tot}}$ is finite, one has to use that $\widetilde{\varphi}^+$ is finite, and then argue as in the proof of Corollary \ref{cor:reduced}.
\end{proof}

\subsection{The non-affine case}

\begin{lemma}
\label{lemma:LocalizationFormal}
Let $A$ be a reduced special $R$-algebra which is integrally closed in the semi-affinoid $k$-algebra  $A\otimes_Rk$.
Let $\Uf$ be a formal open affine subset of $\Xf =\Spf(A)$.
Then  $\Gamma(\Uf,\O_{\Xf})$ is also a reduced special $R$-algebra which is integrally closed in the semi-affinoid $k$-algebra $ \Gamma(\Uf,\O_{\Xf}) \otimes_Rk$ and 
$\Gamma(\Uf,\O_{\Xf}) \simeq \Gamma( \spe_{\Xf}^{-1} (\Uf),  \O^\circ_{\Xf_\eta})$.
\end{lemma}

\begin{proof}
Let us first assume that $\Uf$ is a principal formal open subset of the form $\Uf = \Df(f)$ for some $f\in A$. 
Let $J$ be an ideal of definition of $A$.
Then  $\Gamma(\Uf,\O_{\Xf}) \simeq  A_{\{f\}}$ where $A_{\{f\}} \simeq \widehat{A_f}$ is the completion of the localization $A_f$ with respect to the ideal $J A_f$.
The composition morphism 
$A \to A_f \to \widehat{A_f}$ is regular.
Indeed $A \to A_f$ is regular since it is a localization, and $A_f \to \widehat{A_f}$ is regular since it is the completion of an excellent ring.
It follows that the morphism $A \to \widehat{A_f}$ is regular since regular morphisms are stable under composition \cite[6.8.3]{GroEGAIV2}. 
Thanks to \cite[6.14.4]{GroEGAIV2}, we conclude that $\Gamma(\Uf,\O_{\Xf}) \simeq  \widehat{A_f}$ is integrally closed in $\Gamma(\Uf,\O_{\Xf}) \otimes_R k$.

Let us now assume that $\Uf$ is an arbitrary formal open affine subset of $\Xf$.
It follows from \cite[$\S$ 7 p.74--75]{deJongCryst} that $ \Gamma(\Uf,\O_{\Xf})$ is a special $R$-algebra.
Moreover, we can cover $\Uf$ by some principal formal  open subsets of the form $\Df(f_i)$ for some $f_i \in A$. 
Let us now consider an element $g\in \Gamma(\Uf,\O_{\Xf}) \otimes_R k$ which is  integral over $\Gamma(\Uf,\O_{\Xf}) $. 
This means that $g$ satisfies an equation 
\begin{equation}
\label{equation open integral}
g^d + b_{d-1}g^{d-1} + \ldots + b_1 g +b_0 =0
\end{equation} 
for some $b_j \in \Gamma(\Uf,\O_{\Xf}) $.
But for each principal formal open subset  $\Df(f_i)$, we can restrict the equation \eqref{equation open integral} to 
$\Gamma(  \Df(f_i), \mathcal{O}_{\Xf})\otimes_R k$.
We then get that $g|_{\Df(f_i)}  \in  \Gamma(  \Df(f_i), \mathcal{O}_{\Xf})\otimes_R k$ is integral over  $\Gamma( \Df(f_i), \mathcal{O}_{\Xf})$, 
hence by the first part of the proof, $g|_{\Df(f_i)} \in \Gamma(  \Df(f_i), \mathcal{O}_{\Xf})$.
Since the $\Df(f_i)$ form a covering of $\Uf$, we deduce that $g \in \Gamma(\Uf,\O_{\Xf})$, which proves that $\Gamma(\Uf,\O_{\Xf})$  is integrally closed in $\Gamma(\Uf,\O_{\Xf}) \otimes_R k$.
Finally, the equality $\Gamma(\Uf,\O_{\Xf}) \simeq \Gamma( \spe_{\Xf}^{-1} (\Uf),  \O^\circ_{\Xf_\eta})$ then follows from Theorem \ref{theo:SpecialIntegral}.
\end{proof}

The following result extends \cite[7.2.6.3]{BGR} from affinoid to semi-affinoid $k$-algebras.

\begin{cor}
\label{corollary:reduction open}
Let $X$ be a semi-affinoid $k$-analytic space.
Let $f\in \Gamma(X,\mathcal{O}^\circ_X)$ and let $Y = \{ x\in X \st |f(x)|=1 \}$.
Then $Y$ is a semi-affinoid $k$-analytic space and  the associated map $\widetilde{Y} \to \Xt$ is 
the Zariski open embedding of the principal open subset $D(\widetilde{f}) \subset \Xt$.
\end{cor}

\begin{proof}
It follows from general properties of semi-affinoid $k$-analytic spaces that the inclusion $Y \to X$ is induced by a morphism of semi-affinoid $k$-algebras (see  \cite[Proposition 7.2.1 a)]{deJongCryst}).
Using Remark \ref{remark reduction reduce commutes}, we can easily assume that $X$ and $Y$ are reduced.
Let $\A$ be the reduced semi-affinoid $k$-algebra  associated with  $X$.
By \cite[2.11]{KappUni}, $\Ao$ is a special $R$-algebra. 
Let $\Xf \coloneqq \Spf(\Ao)$.
Let $\mathfrak{D}(f) \subset \Xf$ be the principal formal open subset associated with $f$.
Hence $Y = \spe_{\Xf}^{-1} ( \mathfrak{D}(f))$. 
By Lemma \ref{lemma:LocalizationFormal}, we have $\Gamma(\Df(f), \mathcal{O}_\Xf) \simeq \Bo$ where $\B$ is the semi-affinoid $k$-algebra associated with $Y$.
Let $\mathfrak{J}$ be the biggest ideal sheaf of definition of $\Xf$ (see \cite[10.5.4]{EGA1} for the definition and the properties of $\mathfrak{J}$).
By \cite[10.5.2]{EGA1} there is an isomorphism of schemes $(\Xf, \mathcal{O}_\Xf / \mathfrak{J}) \simeq \Spec(\At) = \Xt$.
Moreover, by \cite[10.5.5]{EGA1}, $\mathfrak{J}|_{\Uf}$ is also the biggest ideal sheaf of definition of the formal scheme $\Uf$. 
Hence $\Gamma(\Uf, \mathfrak{J})$ is the biggest ideal of definition of the adic ring 
$\Gamma(\Uf, \mathcal{O}_\Xf) \simeq \Bo$, and as a consequence $\Gamma(\Uf, \mathfrak{J}) \simeq \Bh$ (by \cite[Remark 2.8]{KappUni}).
Hence we can conclude since
\[\Bt =\Bo / \Bh \simeq \Gamma(\Df(f), \mathcal{O}_\Xf/\mathfrak{I}) \simeq \At[\widetilde{f}^{-1}]\]
\end{proof}

\begin{lemma}
\label{lemma:FormalSaturated}
Let $\Xf$ be a special formal scheme over $R$. The following are equivalent.
\begin{enumerate}[(i)]
\item Any formal open affine subscheme of $\Xf$ is isomorphic to $\Spf(A)$ where $A$ is a reduced special $R$-algebra integrally closed in $A\otimes_Rk$.
\item There exists a covering by formal open affine subschemes $\Uf_i=\Spf(A_i)$ where for each $i$, 
$A_i$ is a reduced special $R$-algebra which is integrally closed in $A_i\otimes_Rk$.
\end{enumerate}
\end{lemma}
\begin{proof}
That (i) implies (ii) is clear. 
To prove the converse implication, let $\mathfrak{U}$ be a formal open affine subset of $\Xf$.
Then $\Uf$ is covered by finitely many  $\Uf\cap \Uf_i$. 
For each $i$ we can find a finite covering $\{\Uf_{i,j}\}$ of $\Uf\cap \Uf_i$ by formal open affine subsets. 
Thanks to Lemma \ref{lemma:LocalizationFormal}, all the $\Uf_{i,j}$ satisfy the expected property. 
We are then reduced to the situation where $\Xf$ is affine and is covered by finitely many formal open affine $\Uf_i$'s which all satisfy the expected property. 
Let us then consider a function $f$ in the integral closure of $\Gamma(\Uf, \O_{\Xf})$ in  $\Gamma(\Uf, \O_{\Xf})\otimes_Rk$. 
Then for all $i$, $f|_{ \Uf_i}$ is in the integral closure of $\Gamma(\Uf_i, \O_{\Xf})$ in  $\Gamma(\Uf_i \O_{\Xf})\otimes_Rk$ which is by assumption $\Gamma(\Uf_i, \O_{\Xf})$.
So $f\in \Gamma(\Uf, \O_{\Xf})$.
\end{proof}

\appendix

\section{Bounded functions on reduced semi-affinoid \texorpdfstring{$k$}{k}-spaces (joint with Christian Kappen)}

In this appendix, we assume that $k$ is a non-trivially discretely valued non-archimedean field, and we let $R$ denote its valuation ring.
We fix a reduced semi-affinoid $k$-algebra $\A$, and we let $X$ denote the associated rigid analytic $k$-space.

\subsubsection*{Projective limits}
Let us start with a reminder on derived functors of projective limits of abelian groups.  
Let 
\[ \cdots \xrightarrow[]{\sigma_{3}}  G_2 \xrightarrow[]{\sigma_{2}} G_1 \xrightarrow[]{\sigma_{1}} 
G_0\]        
be a projective system of abelian groups indexed by $\N$, and let  
\[ 
\begin{array}{cccc}
 \varphi : & \displaystyle{\prod_{n\in \N} G_n} & \to & \displaystyle{\prod_{n\in \N} G_n} \\
          & (a_n)_n & \mapsto & (a_n - \sigma_{n+1}(a_{n+1}))_n.
\end{array}
\]
Then $\displaystyle \ker \varphi \simeq \varprojlim G_n$. 
Let us denote by $\varprojlim^i$ the $i$-th derived functor of  $\varprojlim$. 
According to \cite[3.5.4]{Weib}, one has the following descriptions:  
\begin{equation*}
\begin{array}{rclc}
\label{eq:phi} \varprojlim^1 G_n &  \simeq & \coker \varphi &   \\
\varprojlim^i G_n & = & 0 & \text{for} \ i>1\quad. 
\end{array}
\end{equation*} 

\subsubsection*{A flatness result}
Let us fix a presentation
$\A\simeq (k\otimes_R \Rmn)/I$ 
and let us equip $\A$ with the associated $k$-Banach algebra norm  $\| \ \|$  as in \cite[1.2.5]{KappThesis}. 
For a real number $\varepsilon \in \rK$ 
such that $0<\varepsilon <1$, we set 
\[X_\varepsilon :=\{x\in X \st |S_i(x)|\leq \varepsilon, \ i=1\ldots n\}\quad.\]
Then $X_\varepsilon$ is an affinoid $k$-space (depending on the chosen presentation of $\A$), and we let 
$\A_\varepsilon$ denote the associated affinoid $k$-algebra. It comes with a natural presentation 
\begin{equation*}
\A_\varepsilon \simeq k\langle T_1,\ldots, T_m, \varepsilon^{-1} S_1,\ldots, 
\varepsilon^{-1} S_n \rangle /I
\end{equation*}
and with an associated $k$-Banach algebra norm  $\| \ \|_\varepsilon$ such that if $\varepsilon \leq \varepsilon'$, the restriction morphism 
$\A_{\varepsilon'} \to \A_\varepsilon$ 
is contractive, that is to say, for $f\in \A_{\varepsilon'}$ we have  $\| f\|_{\varepsilon} \leq \| f\|_{\varepsilon'}$.
\par  
Let us now fix an increasing sequence of positive real numbers $(\varepsilon_n)_{n\in \N}$ such that 
$ \displaystyle \lim_{n\to \infty}\varepsilon_n =1$ and such that for all $n\in \N$, we have that $\varepsilon_n \in \rK$.
We set 
$\A_n := \A_{\varepsilon_n}$.
We denote by $X_n$ the $k$-affinoid space associated to $\A_n$.
For $n\in \N$ we denote by $\tau_n \colon \A \to \A_n$ the associated  canonical map, 
and for $m\geq n \in \N$ we denote by  $\sigma_{m,n} \colon \A_m \to \A_n$ the restriction morphism.
By the above remark, each $\sigma_{m,n}$ is a contractive morphism with respect to the norms $\| \  \|_{\varepsilon_m}$ and $\| \  \|_{\varepsilon_n}$.
The sequence $(\A_n)_{n\in \N}$ and the restriction morphisms form a projective system of abelian groups. 
The following statement is inspired by Satz 2.1 \cite{BoschBemerk} , and the proof is verbatim the same as in \cite{BoschBemerk}. 
For the convenience of the reader, we recall it.

\begin{lemma}
\label{lemma:lim} In the above setting, ${\varprojlim}^1 \A_n =0$.
\end{lemma}
\begin{proof}
We have to show that the map
\[ \begin{array}{cccc}
\varphi : &\displaystyle \prod_{n \in \N} \A_n & \to & \displaystyle  \prod_{n \in \N} \A_n \\
          & (a_n)_n & \mapsto & (a_n - \sigma_{n+1,n}(a_{n+1}))_n
\end{array}
\]
is surjective. 
So let us consider a sequence $\displaystyle (g_n) \in \prod_{n\in \N} \A_n$, and let us find a sequence 
$\displaystyle (f_n) \in \prod_{n\in \N} \A_n$ satisfying the conditions
\begin{equation}
\label{eq:system1}
g_n = f_n - \sigma_{n+1,n}(f_{n+1}) \  \ \forall n\in \N.
\end{equation}
\par 
\emph{Step 1}. 
Let us first assume that for all $n\in \N$, $g_n$ lies in the image of the restriction map $\tau_n \colon \A \to \A_n$.
For each $n$, let us then choose $G_n \in \A$ such that 
$g_n  = \tau_n(G_n)$.
We define inductively a sequence $F_n \in \A$ via
\[ \begin{cases}
F_0 := 0 & \\
F_{n+1} := F_n - G_n & \text{for} \ n\geq 0\quad.
\end{cases} \]
Setting $f_n :=\tau_n(F_n)$, we obtain a solution $(f_n)_n$ of \eqref{eq:system1}. \par 
\emph{Step 2}. 
Let us now pick some arbitrary $(g_n) \in \prod_n \A_n$.
For each $n\in \N$, the image of $\A$ in $\A_n$ is dense with respect to the topology induced by $\| \ \|_{\varepsilon_n}$.
Hence for all $n\in \N$, there exists $h_n \in \A$ such that 
$\| g_n - \tau_n(h_n) \|_{\varepsilon_n} \leq 2^{-n}$.
For $n\in \N$, we have $g_n = \tau_n(h_n) + ( g_n -\tau_n(h_n))$. 
By step $1$, there exists $(H_n)_{n\in \N} \in \prod_n \A_n$ such that  $\varphi((H_n)_{n\in \N}) = (\tau_n(h_n))_{n\in \N}$.
Hence it remains to prove that $ ( g_n -\tau_n(h_n))_{n\in \N} \in \im(\varphi)$.
Replacing $g_n$ by $ g_n -\tau_n(h_n)$, we can thus assume that 

\begin{equation*}
\label{eq:bound}
\|g_n\|_{\varepsilon_n}  \leq 2^{-n} \hspace{10pt} \forall n\in \N.
\end{equation*}
Since the morphisms $\sigma_{m,n}$ are contractive, for each $m\geq n$ we have $\| \sigma_{m,n}(g_m) \|_{\varepsilon_n} \leq \| g_m\|_{\varepsilon_m} \leq 2^{-m}$.
Hence since for each $n\in \N$, since $\A_n$ is a $k$-Banach algebra, it makes sense to define
\begin{equation*}
\label{eq:fn}
 f_n :=  \sum_{m \geq n} \sigma_{m,n}(g_m).
\end{equation*}
Finally we have 
\begin{multline*}
f_{n} - \sigma_{n+1,n}(f_{n+1}) = 
\sum_{m \geq n} \sigma_{m,n}(g_m) - \sigma_{n+1,n}\left(  \sum_{m \geq n+1} \sigma_{m,n+1}(g_m) \right) = \\
\sum_{m \geq n} \sigma_{m,n}(g_m) -  \sum_{m \geq n+1} \sigma_{m,n}(g_m) = g_n
\end{multline*}
which proves that $\varphi( (f_n)_n) = (g_n)$.
\end{proof}

For any $\A$-module $M$, we set $M_n := M\otimes_{\A} \A_n$.
\begin{defi}
We let $\Theta$ denote the functor
\[
\begin{array}{cccc}
\Theta: & \{ \text{finitely generated} \ \A\text{-modules} \}& \to & \{ \Gamma(X,\O_X)\text{-modules} \} \\
        & M & \mapsto &\displaystyle  \varprojlim_n M_n\quad.
\end{array}
\]
\end{defi}

The statement and the proof of the following result are again copied almost verbatim from \cite{BoschBemerk} (see however Remark \ref{remark finitely generated} below).

\begin{lemma}
\label{lemma:tau}
The functor $\Theta$ has the following properties:
\begin{enumerate}[(i)]
\item The functor $\Theta$ is exact.
\item For any finitely generated $\A$-module $M$, there is a natural isomorphism 
\[\tau: M\otimes_{\A} \Gamma(X,\O_X) \simeq \Theta(M).\]
\end{enumerate}
\end{lemma}
\begin{proof}
Let us first show that $\Theta$ is exact. By the local theory of uniformly rigid spaces as developed in \cite{KappThesis}, the rings $\A_n$ are flat over $\A$. 
Hence, it suffices to show that $\varprojlim^1 M_n$ vanishes for all finitely generated $\A$-modules $M$. 
Let thus $M$ be a finitely generated $\A$-module, and let 
$F \to M \to 0$ be a finite presentation of $M$. Since the higher derivatives of the projective limit functor vanish, $\varprojlim^1 F_n$ maps onto $\varprojlim^1 M_n$. Since $\varprojlim^1$ commutes with finite direct sums, Lemma \ref{lemma:lim} shows that $\varprojlim^1F_n =0$. The claim follows.
Let us now prove the second statement. Since $\GOX$ is naturally isomorphic to $\varprojlim \A_n$, one has a natural morphism 
\begin{equation*}
\label{nat}
\tau: M\otimes_{\A} \GOX \to \varprojlim M_n.
\end{equation*}
By definition, $\tau$ is an isomorphism when $M=\A$, and more generally $\tau$ is an isomorphism when $M$ is finite and free.
In general, since $\A$ is Noetherian, and $M$ is finitely generated, there is an exact sequence 
 \begin{equation*}
 \label{resolution}
 F_1 \to F_2 \to M \to 0
 \end{equation*}
where $F_1$ and $F_2$ are finite free $\A$-modules.
 Using exactness of $\Theta$, one obtains the exact diagram
 \[
 \xymatrix{
 \ar[r]& F_2\otimes_{\A} \GOX \ar[d]^\wr \ar[r] &  
 F_1\otimes_{\A} \GOX \ar[d]^\wr \ar[r]  &  
 M\otimes_{\A} \GOX \ar[r] \ar[d]^\tau & 0 \\
 \ar[r] & \Theta(F_2) \ar[r]& \Theta(F_1) \ar[r] & \Theta(M) \ar[r] & 0 \quad,
 }
 \]
and it follows that $\tau$ is an isomorphism.
\end{proof}

\begin{rem}
\label{remark finitely generated}
The statements of Lemma \ref{lemma:tau} do not hold for general $\A$-modules. 
Likewise, Satz 2.1 and Korollar 2.2 from \cite{BoschBemerk} do not hold for general $A$-modules either
(although this is not explicitly mentioned in \cite{BoschBemerk}).
Indeed, in example \ref{counterex} below, we give a counterexample to Satz 2.1 and Korollar 2.2 \cite{BoschBemerk} involving modules which are not finitely generated. We want to stress that the results of \cite[Section 2]{BoschBemerk} are not affected by this observation: using the notations of \cite{BoschBemerk}, a correct replacement of \cite[Satz 2.1 and Korollar 2.2]{BoschBemerk} is to say that the functor $\theta$ is exact on the category of finitely generated $A\langle \zeta\rangle$-modules and that $\tau$ is an isomorphism for finitely generated $A\langle \zeta\rangle$-modules.
\end{rem}

\begin{example}\label{counterex}
\label{counterexample theta exact}
We use the notations of Satz 2.1 and Korollar 2.2 of \cite{BoschBemerk}, and we consider $A=k$. 
Moreover, we assume that $\zeta = (\zeta_1)$; that is, $\zeta$ is made of only one variable.
So $\theta$ is the functor sending a $\kX$-module $M$ to the $\kXX$-module
\[\theta(M)= \varprojlim_{n\in \N} M \otimes_{\kX} \knX\;,\]
and for each $M$, $\tau=\tau_M$ is the natural map 
\[ \tau_M : M\otimes_{\kX} \kXX \to \theta(M)\;.\]
Let us write $\A=k\langle\zeta\rangle$ and $\A_j=k\langle\varepsilon_j^{-1}\zeta\rangle$, and let us consider the $\A$-modules  
\[ 
M'=\bigoplus_{j\in\N} \A, \hspace{30pt}
M=\bigoplus_{j\in\N} \A_j , \hspace{30pt}
M''=\bigoplus_{j\in\N} (\A_j/\A)
\]
which form a natural short exact sequence of $\A$-modules
\[ 0 \to M' \to M \to M''\to 0.\]
We denote be $(e_j)_{j\in \N}$ the canonical basis of $M$. 
For simplicity of notations, we also denote by $(e_j)_{j\in \N}$ the canonical bases of $M'$ and $M''$.
We claim that both the natural morphism $\tau_{M''}$ for $M''$ and the induced map $\theta(M)\rightarrow \theta(M'')$ are not surjective, contrary to the statements of Satz 2.1 and Korollar 2.2 of \cite{BoschBemerk}. To this end, let us choose, for each $j\in \N$, a function $f_j \in \A$ such that $f_j$ is invertible in $\A$ and such that for any $m> j$, 
$f_j$ is not invertible in $\A_m$, by picking an element $t\in k$ as well as some positive integers $a,b$ such that 
\[ \varepsilon_{j+1} \geq \left|t\right|^{\frac{a}{b}} > \varepsilon_j\]
and by setting $f_j:= \zeta^b -t^a\in\A=k\langle\zeta\rangle$. Let us now consider the element $g\in\prod_nM''_n$ which is defined by giving, for each $n$, the element
\[
g_n\,:=\,\sum_{j=0}^{n-1}[f_j^{-1}\otimes 1]\cdot e_j
\]
of the $\A_n$-module
\[
M_n''\,=\,\bigoplus_{j\in \N}(\A_j/\A)\otimes_\A\A_n  \, \cong 
\,\bigoplus_{j\in \N}(A_j\otimes_\A\A_n)/\A_n,
\]
where $[\cdot]$ denotes the formation of the residue class and where we have used flatness of $\A_n$ over $\A$ to establish the above isomorphism. Then
\[
g\in\varprojlim_{n\in\N}M''_n
\]
because in $M_n''$, we have that $[f_n^{-1}\otimes 1]=[1\otimes f_n^{-1}]=0$. Let us now consider the natural map
\[
\tau_{M''}\,:\,M''\otimes_\A\varprojlim_n\A_n\rightarrow \varprojlim_n M''_n\;.
\]
For each element $h$ in the image of this map, there exists a $j_0$ such that for all $j>j_0$, the $j$th component of $h_n$ is zero for all $n$. On the other hand, for each $n$, all of the summands $[f_j^{-1}\otimes 1]\cdot e_j$ with $j<n$ defining $g_n$ are nonzero. 
Indeed, if there was an element $h\in\A_n$ with 
\[
f_j^{-1}\otimes 1=1\otimes h\textup{  in  }\A_j\otimes_{\A}\A_n\;,
\]
then the same equality would hold in the completed tensor product, which is $\A_n$, and $f_j^{-1}$ would thus extend to $\A_n$, which is not the case. We have shown that $g\notin\im\tau_{M''}$ and, thus, established our claim that $\tau_{M''}$ is not surjective. The same statement regarding the structure of $g$ shows our second claim, namely that $g$ does not lie in the image of the natural map $\theta(M)\rightarrow\theta(M'')$. 
Indeed, it suffices to remark that
\begin{eqnarray*}
\theta(M)&=&\varprojlim_n\left(\bigoplus_j(\A_j\otimes_\A\A_n)\right)\,=\,\varinjlim_i\varprojlim_n\left(\bigoplus_{j\leq i}(\A_j\otimes_\A\A_n)\right)\\
&=&\bigoplus_j\left(\varprojlim_n(\A_j\otimes_\A\A_n)\right)\;,
\end{eqnarray*}
which follows from the fact that the transition maps $\A_{n+1}\rightarrow\A_n$ are injective and that the $\A_j$ are flat over $\A$.
\end{example}
\begin{prop}
\label{prop:flatness}
 The $\A$-module $\GOX$ is faithfully flat.
\end{prop}

\begin{proof}
Flatness follows from Lemma \ref{lemma:tau}. 
For faithfully flatness, let us consider a maximal ideal $\mathfrak{m}$ of $\A$.
By the Nullstellensatz for semi-affinoid $k$-algebras, $\mathfrak{m}$ corresponds to a rigid point $x$ of $X$.
Since the $X_n$ cover $X$, for $n$ big enough, one has 
$x \in X_n$. 
Hence $\mathfrak{m}':= \{ f\in \GOX \st f(x)=0\}$ is a maximal ideal of $\GOX$ such that 
$\mathfrak{m}' \cap \A = \mathfrak{m}$.
\end{proof}

\begin{lemma}\label{helperlem}
The following are equivalent:
\begin{enumerate}[(i)]
\item The semi-affinoid $k$-algebra $\A$ is normal.
\item The special $R$-algebra $\Ao$ of power-bounded functions in $\A$ is normal.
\end{enumerate}
\end{lemma}

\begin{proof}
Let us first remark that $\Quot(\A) = \Quot(\Ao)$. 
Let us first prove that (i) $\Rightarrow$ (ii).
Let $f\in \Quot(\Ao)$ be an element which satisfies an equation 
$f^n + \sum_{i=0}^{n-1} a_i f^i=0$ with $a_i\in \Ao$. 
Then $f\in \A$ by (i), and since the $a_i$'s $\in \Ao$, it follows that 
$f\in \Ao$.
Let us now prove that (ii) $\Rightarrow$ (i). 
Let $f\in \Quot(\A)$ be an element which satisfies an equation 
$f^n + \sum_{i=0}^{n-1} a_i f^i=0$ with $a_i\in \A$.
Then there exists an integer $m$ such that for all $i$,
$\pi^ma_i \in \Ao$.
Using the same argument as in the proof of Proposition \ref{prop:finite}, 
it follows that $\pi^mf$ satisfies 
a unitary equation with coefficients in $\Ao$,
hence $\pi^mf \in \Ao$ by (ii), hence $f\in \A$.
\end{proof}

\begin{theo}
\label{theo:bound}
Let $\A$ be a reduced semi-affinoid $k$-algebra, and $X$ its associated rigid analytic $k$-space. 
Then $\A \simeq \{f \in \GOX \st \supn{f}< \infty \}$.
\end{theo}

\begin{proof}
If $\A$ is normal, then $\Ao$ is a normal special $R$-algebra by Lemma \ref{helperlem}, and \cite[Theorem 7.4.1]{deJongCryst} shows that 
$\Ao \simeq \Gamma(X, \O^\circ_X)$.
The theorem follows in that case.
 In general, let $\B$ denote the normalization of $\A$. According to \cite{Val75,Val76}, 
$\B$ is a semi-affinoid $k$-algebra.
Let $X'$ denote the rigid analytic $k$-space associated to $\B$, 
let us write $p : X' \to X$ to denote the induced morphism, and let us consider the induced commutative diagram
\[
\xymatrix{
\A \ar[r] \ar[d] & \B \ar[d]  \\
\Gamma(X,\O_X) \ar[r]_{p^*} & \Gamma(X', \O_{X'} )\;.
} \] 
\par 
Let us first observe that
\[\Gamma(X', \O_{X'}) \simeq  \Gamma(X,\O_X) \otimes_{\A} \B\;.\]
Indeed, since $\A \to \B$ is finite, for each $n$ the morphism 
$\A_n \to \A_n \otimes_{\A} \B$ is also finite. 
Since $\A_n$ is a $k$-affinoid algebra, it follows that $\A_n \otimes_{\A} \B$ is also a $k$-affinoid algebra \cite[6.1.1.6]{BGR}.
If $X'_n$ denotes the $k$-affinoid space associated to $\A_n \otimes_{\A} \B$, then $X'_n$ is an affinoid domain of $X'$ and the $X'_n$ cover $X'$.
It follows that 
\begin{align*}
\Gamma(X', \O_{X'})  & \simeq   \varprojlim_{n\in \N} \Gamma(X'_n, \O_{X'}) \\
         & \simeq \varprojlim_{n\in \N} \left( \A_n \otimes_{\A} \B \right) \\
         & \simeq \Gamma(X,\O_X) \otimes_{\A} \B
\end{align*} 
where the second equality follows from the fact that $X'_n$ is a $k$-affinoid space, and the third equality follows from 
Lemma \ref{lemma:tau} (2). 
\par 
Let now $f\in \Gamma(X,\O_X)$ be a bounded function. 
Then 
$p^*(f) \in \Gamma(X',\O_{X'}) \simeq \Gamma(X,\O_X) \otimes_{\A} \B$ is also bounded on $X'$, 
and according to what we have shown in the first part of the proof, 
$p^*(f)$ comes from an element  $b\in \B$. 
\par 
Finally, $\A\subset \B$ is a sub $\A$-module of $\B$ 
because $\A$ is reduced.
Since $\Gamma(X,\O_X)$ is flat over $\A$ (Proposition \ref{prop:flatness}), it follows that 
$\Gamma(X,\O_X)$ is a submodule of 
$\Gamma(X',\O_{X'}) \simeq \Gamma(X,\O_X) \otimes_{\A} \B $.
Since $\Gamma(X,\O_X)$ is even faithfully flat over $\A$, we conclude that
\[ \Gamma(X,\O_X)\cap  \B  = \A\;.\]
Indeed, this follows from \cite[I 3.5 Prop 10(ii)]{BouComAlg}, which 
asserts that if $C \to C'$ is a faithfully flat ring morphism, if
$N$ is a $C$-module and if $N' \subset N$ is a sub $C$-module, 
then $(N'\otimes_C C' )\cap N = N'$. Since $f\in \Gamma(X,\O_X) \cap \B$, it follows that $f\in \A$.
\end{proof}

\begin{example}
\label{remark:counterexample bounded functions}
Theorem \ref{theo:bound} does not hold if we do not assume  the semi-affinoid $k$-algebra $\A$ to be reduced. 
For instance, as in Example \ref{remark example nonreduced1}, let $\A \coloneqq R\llbracket T_1,T_2 \rrbracket /( T_2^2)\otimes_Rk$, and let $X$ be the associated rigid $k$-space. 
Let $f(T_1) \in k\llbracket T_1 \rrbracket$ be a formal power series which converges on the open unit disc, 
but such that $f(T_1) \notin R\llbracket T_1 \rrbracket \otimes_R k$.
Then $f(T_1) T_2 \in \Gamma(X, \mathcal{O}^\circ_X) \setminus \A$.
\end{example}

\bibliographystyle{plain}
\bibliography{bibli}
\end{document}